 \documentclass[final,3p,times]{elsarticle}

 \usepackage{epsfig}

\usepackage{amssymb}
 \usepackage{amsthm}
\usepackage{amsmath,amssymb,amsopn,amsfonts,mathrsfs,amsbsy,amscd}

\usepackage{longtable}
\usepackage{caption}
\usepackage{multirow}

 \usepackage{lineno}

\newcommand{\s}{{\mathfrak{sol}}}

\newcommand{\prs}{\langle\;,\;\rangle}
\newcommand{\too}{\longrightarrow}

\newcommand{\nil}{\mathrm{Nil} }
\newcommand{\sol}{\mathrm{Sol} }

\newcommand{\esp}{\quad\mbox{and}\quad}
\newcommand{\ou}{\quad\mbox{or}\quad}

\newcommand{\G}{{\mathfrak{g}}}

\newcommand{\h}{{\mathfrak{h} }}
\newcommand{\n}{{\mathfrak{n} }}
\newcommand{\ad}{{\mathrm{ad}}}
\newcommand{\tr}{{\mathrm{tr}}}

\newcommand{\B}{{\cal B}}

\newcommand{\di}{\displaystyle}

\newcommand{\na}{\nabla}
\newcommand{\wi}{\widetilde}
\newcommand{\al}{\alpha}
\newcommand{\be}{\beta}
\newcommand{\ga}{\gamma}
\newcommand{\Ga}{\Gamma}
\newcommand{\e}{\epsilon}

\newcommand{\la}{\lambda}
\newcommand{\De}{\Delta}

\font\bb=msbm10

\def\Z{\hbox{\bb Z}}
\def\B{\hbox{\bb B}}
\def\R{\hbox{\bb R}}

\newtheorem{theo}{Theorem}[section]
\newtheorem{pr}{Proposition}[section]

\newtheorem{exem}{Example}

\begin{document}

\begin{frontmatter}


 
 
\title{Biharmonic and harmonic homomorphisms between Riemannian three dimensional unimodular Lie groups}

 \author[label1,label2]{ Boubekeur Sihem, Mohamed Boucetta, }
 \address[label1]{Universit\'e Cadi-Ayyad\\
  Facult\'e des sciences et techniques\\
  BP 549 Marrakesh Morocco\\e-mail: m.boucetta@uca.ma
  }
 
 \address[label2]{Ecole normale sup\'erieure de Bousaada\\
 	Route d'Alger, Bousaada 28001
 	Algeria.
 \\e-mail: sihemmath@hotmail.com 
 }
 


\begin{abstract} We classify biharmonic and harmonic homomorphisms $f:(G,g_1)\too(G,g_2)$ where $G$ is  a connected and simply connected  three-dimensional unimodular Lie group and $g_1$ and $g_2$ are  left invariant Riemannian metrics.
\end{abstract}

\begin{keyword}Harmonic homomorphisms \sep biharmonic homomorphisms \sep Riemannian Lie groups 
\MSC 53C30 \sep \MSC 53C43 \sep \MSC 22E15


\end{keyword}

\end{frontmatter}







\section{Introduction}\label{section1}
The theory of biharmonic maps is old and rich and has gained a growing interest in the last decade (see \cite{baird,sario} and others). The theory of harmonic maps into Lie groups, symmetric spaces or
homogeneous spaces has been extensively studied related to the 
integrable systems by many mathematicians (see for examples \cite{dai, uhlenbeck, wood}). In particular, harmonic maps of Riemann surfaces into
compact Lie groups equipped with a bi-invariant Riemannian metric are called principal
chiral models and intensively studied as toy models of gauge theory in mathematical physics \cite{zakr}.
In the papers  \cite{park, park2},    harmonic inner automorphisms of a compact semi-simple Lie group endowed with a left invariant Riemannian metric where studied. In \cite{boucetta}, there is a detailed study  of  biharmonic and harmonic homomorphisms between Riemannian Lie groups.\footnote{A biharmonic homomorphism between Riemannian Lie groups is a homomorphism of Lie groups $\phi:G\too H$ which is also biharmonic where $G$ and $H$ are endowed with left invariant Riemannian metrics.}

In this paper, we aim the classification, up to a conjugation by automorphisms of Lie groups, of harmonic and biharmonic maps $f:(G,g_1)\too (G,g_2)$ where  $G$ is a non abelian connected and simply-connected three dimensional unimodular Lie group,  $f$ is an homomorphism of Lie groups  and $g_1$ and $g_2$ are two left invariant Riemannian metrics. 
There are five non abelian connected and simply-connected three-dimensional unimodular Lie groups:  the nilpotent Lie group $\nil$, the special unitary group $\mathrm{SU}(2)$, the universal covering group $\wi{\mathrm{PSL}}(2,\R)$ of the special linear group, the solvable Lie group $\sol$ and the universal covering group $\wi{\mathrm{E}_0}(2)$ of the connected component of the Euclidean group.
 There are our main results:
\begin{enumerate}\item For $\nil$ and $\mathrm{Sol}$ we show that a homomorphism is biharmonic if and only if it is harmonic and we classify completely all the harmonic homomorphisms (see Theorems \ref{theoH}, \ref{theohsol} and \ref{theosolbi}).
	\item For $\wi{\mathrm{E}_0}(2)$ we classify completely all the harmonic homomorphisms (see Theorem \ref{theohE2}). For this group there are biharmonic homomorphisms which are not harmonic and we give a complete classification of these homomorphisms (see Theorem \ref{theobihE2}). To our knowledge, these are the first examples of biharmonic not harmonic homomorphisms between Riemannian Lie groups.
	\item For $\mathrm{SU}(2)$ and $\wi{\mathrm{PSL}}(2,\R)$, we give a complete classification of  harmonic homomorphisms (see Theorems \ref{theohsu2} and \ref{thesl2}). We show that these groups have biharmonic homomorphisms which are not harmonic and we give the first examples of these homomorphisms. For $\mathrm{SU}(2)$ we recover the results obtained in \cite{park, park2} and we complete them.
	
	\end{enumerate}
This work is based on \cite{boucetta},  on the results of \cite{lee} which gave a complete classification of left invariant Riemannian metrics on three dimensional Lie groups and on the description given in \cite{P} of the automorphisms of $\mathrm{SU}(2)$ and $\mathrm{PSL}(2,\R)$.    We proceed by a direct computation  and Proposition \ref{test} is a useful trick which simplified many computations. Our straightforward computations were performed using the software Maple.

This paper is divided into seven sections. In Section 2, we give the tools needed in our study and we devote a section to each one of the five groups.

\section{Preliminaries}\label{section2}

Let $\phi:(M,g)\too(N,h)$ be a smooth map between two Riemannian manifolds with $m=\dim M$ and $n=\dim N$. We  denote by $\na^M$ and $\na^N$ the Levi-Civita connexions associated respectively to $g$ and $h$ and  by $T^\phi N$ the vector bundle over $M$ pull-back of $TN$ by $\phi$. It is  an Euclidean vector bundle and the tangent map of $\phi$ is a bundle homomorphism $d\phi:TM\too T^\phi N$. Moreover, $T^\phi N$  carries a connexion $\na^\phi$ pull-back of $\na^N$ by $\phi$ and there is  a connexion on the vector bundle  $\mathrm{End}(TM,T^\phi N)$ given by
\[ (\na_X A)(Y)=\na_X^\phi A(Y)-A\left(\na_X^MY \right),\quad X,Y\in\Ga(TM), A\in \Ga\left(\mathrm{End}(TM,T^\phi N)\right).\]
The map $\phi$ is called harmonic if it is a critical point of the energy $E(\phi)=\frac12\int_M|d\phi|^2\nu_g$. The corresponding Euler-Lagrange equation for the energy is given by the vanishing of the tension field  
\begin{equation}\label{eqtension} \tau(\phi)=\tr_{g}\na d\phi=\sum_{i=1}^m(\na_{E_i}d\phi)(E_i),     \end{equation}
where $(E_i)_{i=1}^m$ is a local frame of orthonormal vector fields. Note that $\tau(\phi)\in\Ga(T^\phi N)$.
The map $\phi$ is called biharmonic if it is a critical point of the
bienergy of $\phi$  defined by $E_2(\phi)=\frac12\int_M|\tau(\phi)|^2\nu_g$. The corresponding Euler-Lagrange equation for the bienergy is given by the vanishing of the bitension field 
\begin{equation}\label{eqbitension} \tau_2(\phi)=-\tr_{g}(\na^\phi)^2_{.\;,\;.}\tau(\phi)-
\tr_{g}R^N(\tau(\phi),d\phi(\;.\;))d\phi(\;{\bf.}\;)=-\sum_{i=1}^m\left((\na^\phi)^2_{E_i,E_i}\tau(\phi)   +R^N(\tau(\phi),d\phi(E_i))d\phi(E_i)\right),
\end{equation}
where $(E_i)_{i=1}^m$ is a local frame of orthonormal vector fields, $(\na^\phi)^2_{X,Y}=\na^\phi_X\na^\phi_Y-\na^\phi_{\na_X^MY}$ and $R^N$ is the curvature of $\na^N$ given by
\[R^N(X,Y)=\na_X^N\na_Y^N-\na_Y^N\na_X^N-\na_{[X,Y]}^N. \]

Let $(G,g)$ be a Riemannian Lie group, i.e., a Lie group endowed with a left invariant Riemannian metric. If $\G=T_eG$ is its Lie algebra and $\prs_\G=g(e)$ then there exists a unique bilinear map $A:\G\times\G\too\G$ called the Levi-Civita product associated to $(\G,\prs_\G)$ given by the formula:
\begin{equation}\label{lc}
2\langle A_uv,w\rangle_\G=\langle[u,v]^\G ,w\rangle_\G+\langle[w,u]^\G
,v\rangle_\G+\langle[w,v]^\G ,u\rangle_\G.
\end{equation} $A$ is entirely determined by the following properties:
\begin{enumerate}\item for any $u,v\in\G$, $A_uv-A_vu=[u,v]^\G$,
	\item for any $u,v,w\in\G$, $\langle A_uv,w\rangle_\G+\langle v,A_uw\rangle_\G=0$.
\end{enumerate}
If we denote by $u^\ell$ the left invariant vector field on $G$ associated to $u\in\G$ then the Levi-Civita connection associated to $(G,g)$ satisfies $\na_{u^\ell}v^\ell=\left(A_uv \right)^\ell$. The couple $(\G,\prs_\G)$ defines a vector say $U^\G\in\G$ by
\begin{equation}\label{eq1}
\langle U^\G,v\rangle_\G=\tr(\ad_v),\quad\mbox{for any}\; v\in\G.
\end{equation} One can deduce  easily from \eqref{lc} that, for any orthonormal basis $(e_i)_{i=1}^n$ of $\G$, 
\begin{equation}\label{ug} U^\G=\sum_{i=1}^nA_{e_i}e_i. \end{equation}
Note that $\G$ is  unimodular iff $U^\G=0$.

Let $\phi:(G,g)\too (H,h)$ be a Lie group homomorphism  between two Riemannian Lie groups.  The differential $\xi:\G\too\h$ of $\phi$ at $e$ is a Lie algebra homomorphism. 
There is a left  action of $G$ on $\Ga(T^\phi H)$ given by 
\[ (a.X)(b)= T_{\phi(ab)}L_{\phi(a^{-1})} X(ab),\quad a,b\in G, X\in \Ga(T^\phi H).   \]
A section $X$ of $T^\phi H$ is called left invariant if, for any $a\in G$, $a.X=X$. For any left invariant section $X$ of $T^\phi H$, we have for any $a\in G$,
$ X(a)=(X(e))^\ell(\phi(a)).$ 
Thus  the space of left invariant sections is isomorphic to the Lie algebra $\h$.
Since $\phi$ is a homomorphism of Lie groups and $g$ and $h$ are left invariant, one can see easily that $\tau(\phi)$ and $\tau_2(\phi)$ are left invariant and hence $\phi$ is harmonic (resp. biharmonic) iff $\tau(\phi)(e)=0$ (resp. $\tau_2(\phi)(e)=0$).  
Now, one can see easily that
\begin{equation}\label{tension}\begin{cases}\di\tau(\xi):=\tau(\phi)(e)=
U^\xi-\xi(U^\G),\\ \di\tau_2(\xi):=\tau_2(\phi)(e)=-\sum_{i=1}^n\left(B_{\xi(e_i)}B_{\xi(e_i)}\tau(\xi)
+K^H(\tau(\xi),\xi(e_i))\xi(e_i)\right)+B_{\xi(U^\G)}\tau(\xi),\end{cases}
\end{equation}where $B$ is the Levi-Civita product associated to $(\h,\prs_\h)$,
\begin{equation}\label{uxi}U^\xi=\sum_{i=1}^nB_{\xi(e_i)}{\xi(e_i)},\end{equation}  $(e_i)_{i=1}^n$ is an orthonormal basis of $\G$ and $K^H$ is the curvature of $B$ given by $K^H(u,v)=[B_u,B_v]-B_{[u,v]}$. So we get the following proposition.
\begin{pr}\label{pr1}Let $\phi:G\too H$ be an homomorphism between two Riemannian Lie groups. Then $\phi$ is harmonic (resp. biharmonic) iff $\tau(\xi)=0$ (resp. $\tau_2(\xi)=0$), where $\xi:\G\too\h$ is the differential of $\phi$ at $e$.
	
\end{pr}

Thus the study of biharmonic and harmonic homomorphisms between connected and simply-connected Lie groups reduces to the study of their differential so,
through this paper, we consider homomorphisms $\xi:(\G,\prs_1)\too(\G,\prs_2)$ where $\G$ is a Lie algebra and $\prs_1$ and $\prs_2$ are two Euclidean products. We call $\xi$ harmonic (resp. biharmonic) if $\tau(\xi)=0$ (resp. $\tau_2(\xi)=0$). 

The classification of biharmonic and harmonic homomorphisms will be done up to a conjugation. Two homomorphisms between Euclidean Lie algebras $$\xi_1:(\G,\prs_1^1)\too(\G,\prs_2^1)\esp\xi_2:(\G,\prs_1^2)\too(\G,\prs_2^2)$$
are conjugate if there exists two isometric automorphisms $\phi_1:(\G,\prs_1^1)\too(\G,\prs_1^2)$ and $\phi_2:(\G,\prs_2^1)\too(\G,\prs_2^2)$ such that $\xi_2=\phi_2\circ\xi_1\circ\phi_1^{-1}$.

We give now a criteria which will be useful in order to show that an homomorphism is harmonic if and only if it is biharmonic. 

Let $\xi:(\G,\prs_1)\too(\h,\prs_2)$ be an homomorphism. We suppose that $\G$ is unimodular. The following formulas was established in \cite[Proposition 2.4]{boucetta}:
\begin{eqnarray*}\langle \tau(\xi),u\rangle_2&=&\tr(\xi^*\circ\ad_u\circ\xi),\\
	\langle \tau_2(\xi),u\rangle_2&=&\tr(\xi^*\circ(\ad_u+\ad_u^*)\circ\ad_{\tau(\xi)}\circ\xi)-
	\langle[u,\tau(\xi)],\tau(\xi)\rangle_2,
\end{eqnarray*}	where $\xi^*:\h\too\G$ and $\ad_u^*:\h\too\h$ are given by
\[ \langle\xi^*u,v\rangle_1=\langle u,\xi v\rangle_2\esp \langle\ad_u^*x,y\rangle_2=\langle\ad_uy,x\rangle_2,\;x,y,u\in\h,v\in\G. \]
By combining these two formulas, we get
\[ \langle\tau_2(\xi),u\rangle_2=\tr(\xi^*\circ(\ad_u+\ad_u^*)\circ\ad_{\tau(\xi)}\circ\xi)-
\tr(\xi^*\circ\ad_{[u,\tau(\xi)]}\circ\xi). \]
So if $\xi$ is biharmonic then $\tau(\xi)$ is solution of the linear system
\begin{equation} \tr(\xi^*\circ(\ad_u+\ad_u^*)\circ\ad_{X}\circ\xi)-
\tr(\xi^*\circ\ad_{[u,X]}\circ\xi)=0,\;u\in\h. \end{equation}
If $\B=(f_1,\ldots,f_m)$ is a basis of $\h$ this system is equivalent to
\[ M_\xi(\B) X=0, \]where $M_\xi(\B)=(m_{ij})_{1\leq i\leq j\leq m}$ and
\[ m_{ij}=\tr(\xi^*\circ(\ad_{f_i}+\ad_{f_i}^*)\circ\ad_{f_j}\circ\xi)-
\tr(\xi^*\circ\ad_{[f_i,f_j]}\circ\xi). \]
We call $M_\xi(\B)$ the test matrix of $\xi$ in the basis $(f_1,\ldots,f_n)$.
\begin{pr}\label{test} If $\det(M_\xi(\B))\not=0$ then $\xi$ is biharmonic if and only if it is harmonic.
	
\end{pr}

We end this section by describing the main objects of this study, namely, the 3-dimensional unimodular Lie algebras.

They are five unimodular  simply connected three dimensional unimodular non abelian Lie groups: \begin{enumerate}
	
	\item The nilpotent Lie group $\nil$ known as Heisenberg group whose Lie algebra will be denoted by $\mathfrak{n}$. We have
	\[ \nil=\left\{\left(\begin{matrix} 1&x&z\\0&1&y\\0&0&1\end{matrix}\right),x,y,z\in\R   \right\}
	\esp \mathfrak{n}=\left\{\left(\begin{matrix} 0&x&z\\0&0&y\\0&0&0\end{matrix}\right),x,y,z\in\R   \right\}. \]
	 The Lie algebra $\n$ has a basis $\B_0=(X_1,X_2,X_3)$ where
	 \[ X_1=\left(\begin{matrix} 0&1&0\\0&0&0\\0&0&0\end{matrix}\right), X_2=\left(\begin{matrix} 0&0&0\\0&0&1\\0&0&0\end{matrix}\right)\esp X_3=\left(\begin{matrix} 0&0&1\\0&0&0\\0&0&0\end{matrix}\right) \]and
	  where   the non-vanishing Lie brackets are $[X_1,X_2]=X_3$. 
	\item $\mathrm{SU}(2)=\left\{\left(\begin{matrix} a+bi&-c+di\\c+di&a-bi\end{matrix}\right),a^2+b^2+c^2+d^2=1   \right\}
	\esp \mathfrak{su}(2)=\left\{\left(\begin{matrix} iz&y+ix\\-y+xi&-zi\end{matrix}\right),x,y,z\in\R   \right\}.$ The Lie algebra $\mathfrak{su}(2)$ has a basis $\B_0=(X_1,X_2,X_3)$
	\[ X_1=\frac12\left(\begin{matrix}0&i\\i&0\end{matrix} \right),\; X_2=\frac12\left(\begin{matrix}0&1\\-1&0\end{matrix} \right)\esp X_3=\frac12\left(\begin{matrix}-i&0\\0&i\end{matrix} \right) \]
	and where   the non-vanishing Lie brackets are
	\begin{equation*} \label{eqsu2}[X_1,X_2]=X_3,\;[X_2,X_3]=X_1\esp [X_3,X_1]=X_2. \end{equation*} 
	
	\item The universal covering group $\wi{\mathrm{PSL}}(2,\R)$ of $\mathrm{SL}(2,\R)$ whose Lie algebra  is $\mathrm{sl}(2,\R)$. 
	The Lie algebra $\mathrm{sl}(2,\R)$ has a basis $\B_0=(X_1,X_2,X_3)$ where
	\[ X_1=\frac12\left(\begin{array}{cc} 0&1\\1&0   \end{array} \right),\; 
	X_2=\frac12\left(\begin{array}{cc} 1&0\\0&-1   \end{array} \right)
	\esp X_3=\frac12\left(\begin{array}{cc} 0&1\\-1&0   \end{array} \right)\]and where
	 the non-vanishing Lie brackets are
	\begin{equation*} \label{eqsl}[X_1,X_2]=-X_3,\;[X_2,X_3]=X_1\esp [X_3,X_1]=X_2. \end{equation*} 
	\item The solvable Lie group $\sol=\left\{\left(\begin{matrix} e^x&0&y\\0&e^{-x}&z\\0&0&1\end{matrix}\right),x,y,z\in\R   \right\}$ whose Lie algebra is $\mathfrak{sol}=\left\{\left(\begin{matrix} x&0&y\\0&-x&z\\0&0&0\end{matrix}\right),x,y,z\in\R   \right\}$. The Lie algebra $\mathfrak{sol}$ has a basis $\B_0=(X_1,X_2,X_3)$ where
	\[ 
	X_1=\left(\begin{matrix} 0&0&1\\0&0&0\\0&0&0\end{matrix}\right),\;
	X_2=\left(\begin{matrix} 0&0&0\\0&0&1\\0&0&0\end{matrix}\right)\esp X_3=\left(\begin{matrix} 1&0&0\\0&-1&0\\0&0&0\end{matrix}\right) \]and where the non-vanishing Lie brackets are
	\begin{equation*} \label{eqsol} [X_3,X_1]=X_1\esp [X_3,X_2]=-X_2. \end{equation*}
	\item The universal covering group $\wi{\mathrm{E}_0}(2)$ of the Lie group 
	\[ \mathrm{E}_0(2)=\left\{\left(\begin{matrix} \cos(\theta)&\sin(\theta)&x\\-\sin(\theta)&\cos(\theta)&y\\0&0&1\end{matrix}\right),\theta,x,y\in\R   \right\}. \]Its Lie algebra is
	\[ \mathrm{e}_0(2)=\left\{\left(\begin{matrix} 0&\theta&x\\-\theta&0&y\\0&0&0\end{matrix}\right),\theta,y,z\in\R   \right\}.  \]The Lie algebra $\mathrm{e}_0(2)$ has a basis $\B_0=(X_1,X_2,X_3)$ where
	\[ 
	X_1=\left(\begin{matrix} 0&0&1\\0&0&0\\0&0&0\end{matrix}\right),\;
	X_2=\left(\begin{matrix} 0&0&0\\0&0&1\\0&0&0\end{matrix}\right)\esp X_3=\left(\begin{matrix} 0&-1&0\\1&0&0\\0&0&0\end{matrix}\right) \]
	and where the non-vanishing Lie brackets are
	\begin{equation*} \label{eqe2} [X_3,X_1]=X_2\esp [X_3,X_2]=-X_1. \end{equation*}

\end{enumerate}

In Table \ref{1}, we collect the informations on these Lie algebras we will use in the next sections. For each Lie algebra among the five Lie algebras above, we give the set of its homomorphisms and the equivalence classes  of Riemannian metrics carried out by this Lie algebra. These equivalence classes were determined in \cite[Theorems 3.3-3.7]{lee}. For $\n$, $\s$ and $e_0(2)$ the homomorphisms can be determined easily. For $\mathfrak{su}(2)$ and $\mathrm{sl}(2,\R)$ an homomorphism is necessarily an inner automorphism and these were determined in \cite{P}.

\begin{center}
	\begin{tabular}{|c|l|l|l|}
		\hline
		Lie algebra&Non-vanishing Lie brackets&Homomorphisms&Equivalence classes\\&&& of Metrics\\
		\hline
		$\n$&$[X_1,X_2]=X_3$&$\left( \begin{matrix}
		\al_1&\al_2&0\\\be_1&\be_2&0\\\al_3&\be_3&\al_1\be_2-\al_2\be_1
		\end{matrix} \right)$&$\mathrm{Diag}(\la,\la,1)$, $\la>0$\\
		\hline
	\multirow{2}{*}{$e_0(2)$}&\multirow{2}{*}{$[X_3,X_1]=X_2, [X_3,X_2]=-X_1$}&$\left( \begin{matrix}
		0&0&a\\0&0&b\\0&0&\ga
		\end{matrix} \right)$, $\left( \begin{matrix}
		\al&-\be&a\\\be&\al&b\\0&0&1
		\end{matrix} \right),\;\ga^2\not=1$&$\mathrm{Diag}(1,\mu,\nu)$,\\
		&&$\left( \begin{matrix}
		\al&\be&a\\\be&-\al&b\\0&0&-1
		\end{matrix} \right)$&$0<\mu\leq1,\nu>0$\\
		\hline
		\multirow{2}{*}{$\s$}&\multirow{2}{*}{$\;[X_3,X_1]=X_1, [X_3,X_2]=-X_2$}&$\left( \begin{matrix}
		0&0&a\\0&0&b\\0&0&\ga
		\end{matrix} \right)$, $\left( \begin{matrix}
		\al&0&a\\0&\be&b\\0&0&1
		\end{matrix} \right)$&$\left( \begin{matrix}
		1&1&0\\1&\mu&0\\0&0&\nu
		\end{matrix} \right),\;\nu>0,\mu>1$\\
		&&$\left( \begin{matrix}
		0&\be&a\\\al&0&b\\0&0&-1
		\end{matrix} \right),\ga^2\not=1$&$\mathrm{Diag}(1,1,\nu)$\\
		\hline
		\multirow{2}{*}{$\mathrm{sl}(2,\R)$}&$[X_1,X_2]=-X_3,\;[X_3,X_1]=X_2,$&
		\multirow{2}{*}{$\mathrm{Rot}_{xy}.\mathrm{Boost}_{xz}.\mathrm{Boost}_{yz}$}&$\mathrm{Diag}(\la,\mu,\nu)$, \\
		&$[X_2,X_3]=X_1$&&$0<\la\leq \mu$ and $\nu>0$\\
		\hline
		\multirow{2}{*}{$\mathfrak{su}(2)$}&$[X_1,X_2]=X_3,\;[X_3,X_1]=X_2,$&
		\multirow{2}{*}{$\mathrm{Rot}_{xy}.\mathrm{Rot}_{xz}.\mathrm{Rot}_{yz}$}&$\mathrm{Diag}(\la,\mu,\nu),\;$\\
		&$[X_2,X_3]=X_1$&&$0<\nu\leq \mu\leq \lambda$\\
		\hline
	\end{tabular}\captionof{table}{ \label{1}}\end{center}
\[ \mathrm{Rot}_{xy}=\left(\begin{array}{ccc}\cos(a)&\sin(a)&0\\-\sin(a)& \cos(a)&0\\0&0&1 \end{array}   \right),
\mathrm{Rot}_{xz}=\left(\begin{array}{ccc}\cos(a)&0&\sin(a)\\0&1 &0\\
-\sin(a)&0&\cos(a) \end{array}   \right),\mathrm{Rot}_{yz}=\left(\begin{array}{ccc}1&0&0\\0&\cos(a) &\sin(a)\\
0&-\sin(a)&\cos(a) \end{array}   \right) \]
\[ \mathrm{Boost}_{xz}=\left(\begin{array}{ccc}\cosh(a)&0&\sinh(a)\\0&1 &0\\
\sinh(a)&0&\cosh(a) \end{array}\right),\mathrm{Boost}_{yz}=\left(\begin{array}{ccc}1&0&0\\0&\cosh(a) &\sinh(a)\\
0&\sinh(a)&\cosh(a) \end{array}   \right).    \]

In the following sections, the computation of $\tau(\xi)$ and $\tau_2(\xi)$ are performed by the software Maple and all the direct computations as well.
	
	\section{Harmonic and biharmonic homomorphisms on the 3-dimensional Heisenberg Lie group}\label{section3}
	The following result gives a complete classification of harmonic and biharmonic homomorphisms of $\n$.

	\begin{theo}\label{theoH} An homomorphism of $\n$ is biharmonic if and only if it is harmonic. Moreover, it is harmonic if and only if it is conjugate to
		 $\xi:(\n,\prs_1)\too(\n,\prs_2)$ where
			\[ \xi=\left( \begin{matrix}
			\al_1&\al_2&0\\\be_1&\be_2&0\\0&0&\al_1\be_2-\al_2\be_1
			\end{matrix} \right)\;\ou\;\left( \begin{matrix}
			a\be_3&-a\al_3&0\\b\be_3&-b\al_3&0\\\al_3&\be_3&0
			\end{matrix} \right),(\al_3,\be_3)\neq (0,0) \]and
			$\mathrm{Mat}(\prs_i,\B_0)=\mathrm{Diag}(\la_i,\la_i,1)$ and $\la_i>0$, $i=1,2$.
				\end{theo}
	
	\begin{proof} The first part of the theorem is a consequence of \cite[Theorem 6.5]{boucetta}. On the other hand, according to Table \ref{1}, and homomorphism $\xi:(\n,\prs_1)\too(\n,\prs_2)$ has, up to a conjugation, the form
		\[ \xi=\left( \begin{matrix}
		\al_1&\al_2&0\\\be_1&\be_2&0\\\al_3&\be_3&\al_1\be_2-\al_2\be_1
		\end{matrix} \right)\esp \prs_i=\mathrm{Diag}(\la_i,\la_i,1),\;\la_i>0, i=1,2. \]Then
		\[ \tau(\xi)= \frac{\al_3\be_1+\be_3\be_2}{\lambda_2\lambda_1}X_1-\frac{\al_3\al_1+\be_3\al_2}{\lambda_2\lambda_1}X_2 \]and the second part of the theorem follows.
		\end{proof}
	
	\section{Harmonic and biharmonic homomorphisms on $\wi{E_0}(2)$}\label{section4}
	
	The situation on $e_0(2)$ is different and there exists biharmonic homomorphisms which are not harmonic. The following  two theorems give a complete classification of harmonic and biharmonic homomorphisms on $e_0(2)$.
	
	\begin{theo}\label{theohE2} An homomorphism of $e_0(2)$ is harmonic if and only if it it is conjugate to $\xi:(e_0(2),\prs_1)\too(e_0(2),\prs_2)$ where
	$\mathrm{Mat}(\prs_i,\B_0)=\left( \begin{matrix}
		1&0&0\\0&\mu_i&0\\0&0&\nu_i
	\end{matrix} \right),0<\mu_i\leq1,\nu_i>0$	and either
	\begin{enumerate}\item $\xi=\left( \begin{matrix}
		0&0&a\\0&0&b\\0&0&\ga
		\end{matrix} \right)$ with $\ga^2\not=1$ and $\left(a\not=0,b\not=0,\ga=0,\mu_2=1\right),\; \left((a,b)\not=(0,0),ab=0,\ga=0\right)\ou(a=b=0)$
		\item $\xi=\left( \begin{matrix}
		\al&-\be&a\\\be&\al&b\\0&0&1
		\end{matrix} \right)$ and $(a=b=0,\al=0),(a=b=0,\be=0),(a=b=0,\mu_1=1)\;\mbox{or}\;(a=b=0,\mu_2=1)$,
		
		\item $\xi=\left( \begin{matrix}
		\al&\be&a\\\be&-\al&b\\0&0&-1
		\end{matrix} \right)$ and $(a=b=0,\al=0),(a=b=0,\be=0),(a=b=0,\mu_1=1)\;\mbox{or}\;(a=b=0,\mu_2=1)$.
		
		\end{enumerate}

	\end{theo}
	
	\begin{proof} According to Table \ref{1}, and homomorphism $\xi:(e_0(2),\prs_1)\too(e_0(2),\prs_2)$ has, up to a conjugation, the form $\mathrm{Mat}(\prs_i,\B_0)=\mathrm{Diag}(1,\mu_i,\nu_i), i=1,2, 
		0<\mu_i\leq 1, \nu_i>0$ and
		\[ \xi=\left( \begin{matrix}
		0&0&a\\0&0&b\\0&0&\ga
		\end{matrix} \right),\ga^2\not=1, \xi=\left( \begin{matrix}
		\al&-\be&a\\\be&\al&b\\0&0&1
		\end{matrix} \right) \ou \xi=\left( \begin{matrix}
		\al&\be&a\\\be&-\al&b\\0&0&-1
		\end{matrix} \right). \]
		
	$\bullet$	$\xi=\left( \begin{matrix}
		0&0&a\\0&0&b\\0&0&\ga
		\end{matrix} \right)$ with $\ga^2\not=1$. We have
		\[ \di\tau(\xi)= -{\frac {\gamma\,\mu_{{2}}b}{\nu_{{1}}}}X_1+{
			\frac {\gamma\,a}{\mu_{{2}}\nu_{{1}}}}X_2+{\frac {ba \left( \mu_{{2}}-1
				\right) }{\nu_{{2}}\nu_{{1}}}}X_3\]
		and $\tau(\xi)=0$ if and only if
		\[  \left(a\not=0,b\not=0,\ga=0,\mu_2=1\right),\; \left((a,b)\not=(0,0),ab=0,\ga=0\right)\ou(a=b=0).\]
	
	$\bullet$ $\xi=\left( \begin{matrix}
	\al&-\be&a\\\be&\al&b\\0&0&1
	\end{matrix} \right)$. We have
	\[ \tau(\xi)=-\frac{\mu_2b}{\nu_1}X_1+\frac{ a}{\mu_2\nu_1}X_2+
	\frac{(\mu_2-1)(\al\be\nu_1(\mu_1-1)+ab\mu_1)}{\mu_1\nu_1\nu_2}X_3\]	 
		 and $\tau(\xi)=0$ if and only if
		 \[ (a=b=0,\al=0),(a=b=0,\be=0),(a=b=0,\mu_1=1)\;\mbox{or}\;(a=b=0,\mu_2=1). \]
	$\bullet$ $\xi=\left( \begin{matrix}
	\al&\be&a\\\be&-\al&b\\0&0&-1
	\end{matrix} \right)$. We have
	\[ \tau(\xi)=\frac{\mu_2b}{\nu_1}X_1-\frac{ a}{\mu_2\nu_1}X_2+
	\frac{(\mu_2-1)(\al\be\nu_1(\mu_1-1)+ab\mu_1)}{\mu_1\nu_1\nu_2}X_3\]	 
	and $\tau(\xi)=0$ if and only if
	\[ (a=b=0,\al=0),(a=b=0,\be=0),(a=b=0,\mu_1=1)\;\mbox{or}\;(a=b=0,\mu_2=1). \]

		\end{proof}
	
	\begin{theo} \label{theobihE2} An homomorphism of $e_0(2)$   is biharmonic not harmonic if and only if it is conjugate to $\xi:(e_0(2),\prs_1)\too(e_0(2),\prs_2)$ where $\mathrm{Mat}(\prs_i,\B_0)=\left( \begin{matrix}
		1&0&0\\0&\mu_i&0\\0&0&\nu_i
		\end{matrix} \right),0<\mu_i\leq1,\nu_i>0$ and either:
		\begin{enumerate}
			\item $\xi=\left( \begin{matrix}
				0&0&a\\0&0&b\\0&0&0
			\end{matrix} \right)$ and $\left(a^2=b^2,ab\not=0\right),$
		\item $(\mu_1\not=0,\mu_2\not=1)$, $\xi=\left( \begin{matrix}
		\al&-\be&a\\\be&\al&b\\0&0&1
		\end{matrix} \right)$  and
		 $\left(  a=b=0,\al^2=\be^2,\al\be\not=0    \right)$ or
		\[ \di\left(a=\e b\mu_2\sqrt{\mu_1},b\not=0  ,\al^2=\be^2={\frac {\sqrt{\mu_{{1}}} \left( {\mu_{{2}}}^{2}\nu_{{2}}+{a}^{2}(\mu_2-1)^2
				\right) }{\mu_2\,\nu_{{
						1}} \left( \mu_{{2}}-1 \right) ^{2} \left(1- \mu_{{1}} \right) }},\be=\e\al,\e=\pm1.    \right) \]
		\item  $(\mu_1\not=0,\mu_2\not=1)$,  $\xi=\left( \begin{matrix}
		\al&\be&a\\\be&-\al&b\\0&0&-1
		\end{matrix} \right)$ and
		$\left(  a=b=0,\al^2=\be^2,\al\be\not=0    \right)$ or
		\[ \di\left(a=\e b\mu_2\sqrt{\mu_1},b\not=0  ,\al^2=\be^2={\frac {\sqrt{\mu_{{1}}} \left( {\mu_{{2}}}^{2}\nu_{{2}}+{a}^{2}(\mu_2-1)^2
				\right) }{\mu_2\,\nu_{{
						1}} \left( \mu_{{2}}-1 \right) ^{2} \left(1- \mu_{{1}} \right) }},\be=\e\al,\e=\pm1    \right). \]

		\end{enumerate}
	\end{theo}
	
	\begin{proof} As in the proof of Theorem \ref{theohE2}, according to Table \ref{1}, and homomorphism $\xi:(e_0(2),\prs_1)\too(e_0(2),\prs_2)$ has, up to a conjugation, the form
		$\mathrm{Mat}(\prs_i,\B_0)=\mathrm{Diag}(1,\mu_i,\nu_i), i=1,2, 
		0<\mu_i\leq 1, \nu_i>0$ and
		\[ \xi=\left( \begin{matrix}
		0&0&a\\0&0&b\\0&0&\ga
		\end{matrix} \right),\ga^2\not=1, \xi=\left( \begin{matrix}
		\al&-\be&a\\\be&\al&b\\0&0&1
		\end{matrix} \right) \ou \xi=\left( \begin{matrix}
		\al&\be&a\\\be&-\al&b\\0&0&-1
		\end{matrix} \right). \]
		
		$\bullet$	$\xi=\left( \begin{matrix}
		0&0&a\\0&0&b\\0&0&\ga
		\end{matrix} \right)$ with $\ga^2\not=1$. We have
		\[\begin{cases} 
		\di	\tau_2(\xi)=-{\frac {b\gamma \left(  \left( {\gamma}^{2}\nu_
				{{2}}+{a}^{2} \right) {\mu_{{2}}}^{2}-2\,{a}^{2}\mu_{{2}}+{a}^{2}
				\right) }{{\nu_{{1}}}^{2}\nu_{{2}}}}X_1+{\frac {\gamma\,a \left( {
					b}^{2}{\mu_{{2}}}(\mu_2-1)^2+{
					\gamma}^{2}\nu_{{2}} \right) }{{\nu_{{1}}}^{2}{\mu_{{2}}}^{2}\nu_{{2}}
			}}X_2\\\di+{\frac { \left(  \left( {\gamma}^{2}\nu_{{2}}+{a}^{2}-{b}^{2}
			\right) {\mu_{{2}}}^{2}+ \left( {\gamma}^{2}\nu_{{2}}-{a}^{2}+{b}^{2}
			\right) \mu_{{2}}+{\gamma}^{2}\nu_{{2}} \right) b \left( \mu_{{2}}-1
			\right) a}{{\nu_{{1}}}^{2}{\nu_{{2}}}^{2}\mu_{{2}}}}X_3.
	.\end{cases}
	\]If $\ga=0$ then
	\[ \tau_2(\xi)={\frac { \left( a-b \right)  \left( a+b \right)  \left( \mu_{{2}}-1
			\right) ^{2}ba}{{\nu_{{1}}}^{2}{\nu_{{2}}}^{2}}}X_3
	\]and $\xi$ is biharmonic not harmonic if and only if $a^2=b^2$ and $ab\not=0$.
	
	 If $\ga\not=0$ and $b\not=0$ and $\tau_2(\xi)=0$ then
	\[ \left( {\gamma}^{2}\nu_
	{{2}}+{a}^{2} \right) {\mu_{{2}}}^{2}-2\,{a}^{2}\mu_{{2}}+{a}^{2}=0. \]
	The discriminant of this equation on $\mu_2$ is $\De=-4a^2\ga^2\nu_2\leq0$ and this equation has no solution. It is also clear that if $\ga\not=0$ and $a\not=0$ then $\tau_2(\xi)\not=0$. In conclusion $\xi$ is biharmonic not harmonic if and only if 
	\[ \left(\ga=0,a^2=b^2,ab\not=0\right). \]
		
$\bullet$	$\xi=\left( \begin{matrix}
	\al&-\be&a\\\be&\al&b\\0&0&1
	\end{matrix} \right)$. We have
	\[\begin{cases}\di \tau(\xi)=-\frac{\mu_2b}{\nu_1}X_1+\frac{ a}{\mu_2\nu_1}X_2+
	\frac{(\mu_2-1)(\al\be\nu_1(\mu_1-1)+ab\mu_1)}{\mu_1\nu_1\nu_2}X_3,\\
	\di\tau_2(\xi)=A_1X_1+A_2X_2+A_3X_3,\\
	\di A_1=-{\frac {\, \left( b\mu_{{1}} \left( \mu_{{2}}-1 \right) ^{2}{a}
			^{2}+\beta\,\alpha\,\nu_{{1}} \left( \mu_{{2}}-1 \right) ^{2} \left( 
			\mu_{{1}}-1 \right) a+b\mu_{{1}}{\mu_{{2}}}^{2}\nu_{{2}}
			\right) }{\mu_{{1}}{\nu_{{1}}}^{2}\nu_{{2}}}},
	\\
	\di	A_2= {\frac {\, \left(  \left( \alpha\,\nu_{{1}} \left( \mu_{{1}}-1
			\right) \beta+ab\mu_{{1}} \right) b{\mu_{{2}}}^{3}-2\, \left( \alpha
			\,\nu_{{1}} \left( \mu_{{1}}-1 \right) \beta+ab\mu_{{1}} \right) b{\mu
				_{{2}}}^{2}+ \left( \alpha\,\nu_{{1}} \left( \mu_{{1}}-1 \right) \beta
			+ab\mu_{{1}} \right) b\mu_{{2}}+a\mu_{{1}}\nu_{{2}}
			\right) }{{\nu_{{1}}}^{2}{\mu_{{2}}}^{2}\nu_{{2}}\mu_{{1}}}}
	,\\
	\di{\mu_{{1}}}^{2}{\nu_{{1}}}^{2}{\nu_{{2}}}^{2}\mu_{{2}}	A_3= \mu_{{2}}\beta\,{\nu_{{1}}}^{2} \left( \mu_{{1}}-1 \right) ^{2}
	\left( \mu_{{2}}-1 \right) ^{2}{\al}^{3}+\mu_{{2}}\nu_{{1}}ab\mu_{{1}}
	\left( \mu_{{2}}-1 \right) ^{2} \left( \mu_{{1}}-1 \right) {\al}^{2}\\+
	\mu_{{2}}\beta\,\nu_{{1}} \left( \mu_{{1}}-1 \right)  \left( -{\beta}^
	{2}\mu_{{1}}\nu_{{1}}+{a}^{2}\mu_{{1}}-{b}^{2}\mu_{{1}}+{\beta}^{2}\nu
	_{{1}} \right)  \left( \mu_{{2}}-1 \right) ^{2}\al+ab\mu_{{1}} \left( -{
		\beta}^{2}\mu_{{1}}{\mu_{{2}}}^{2}\nu_{{1}}+\mu_{{1}}{\mu_
		{{2}}}^{2}\nu_{{2}}+{a}^{2}\mu_{{1}}{\mu_{{2}}}^{2}\right.\\\left.-{b}^{2}\mu_{{1}}{
		\mu_{{2}}}^{2}+{\beta}^{2}\mu_{{1}}\mu_{{2}}\nu_{{1}}+{\beta}^{2}{\mu_
		{{2}}}^{2}\nu_{{1}}+\mu_{{1}}\mu_{{2}}\nu_{{2}}-{a}^{2}\mu
	_{{1}}\mu_{{2}}+{b}^{2}\mu_{{1}}\mu_{{2}}-{\beta}^{2}\mu_{{2}}\nu_{{1}
	}+\mu_{{1}}\nu_{{2}} \right)  \left( \mu_{{2}}-1 \right) 		
	\end{cases} \]	and the test matrix is given by
	\[ M_\xi(\B_0)=\left( \begin {array}{ccc} {\frac {\mu_{{2}}}{\nu_{{1}}}}
	&0&-{\frac {a\, \left( \mu_{{2}}-1 \right) }{\nu_{{1}}}}
	\\ \noalign{\medskip}0&{\frac {1}{\nu_{{1}}}}&{\frac {
			\,b \left( \mu_{{2}}-1 \right) }{\nu_{{1}}}}
	\\ \noalign{\medskip}-{\frac {\,\mu_{{2}}a}{\nu_{{1}}}}&-{\frac 
		{\,b}{\nu_{{1}}}}&{\frac { \left( \mu_{{2}}-1 \right)  \left( 
			\left(  \left( {\alpha}^{2}-{\beta}^{2} \right) \nu_{{1}}+{a}^{2}-{b}
			^{2} \right) \mu_{{1}}+\nu_{{1}} \left( -{\alpha}^{2}+{\beta}^{2}
			\right)  \right) }{\mu_{{1}}\nu_{{1}}}}\end {array} \right)\esp \det(M_\xi(\B))= {\frac {\mu_{{2}} \left( \mu_{{2}}-1 \right)  \left( 
			\alpha^2-\beta^2 \right)    \left( \mu_{{1}}-1
			\right) }{{\nu_{{1}}}^{2}\mu_{{1}}}}.	
	 \]Note first that if $\mu_2=1$ then $\tau_2(\xi)= -{\frac {b}{{\nu_{{1}}}^{2}}}X_1+
	 {\frac {a}{{\nu_{{1}}}^{2}}}X_2$ and $\xi$ is biharmonic if and only if it is harmonic. If $\mu_1=1$ then
	 \[ A_1=-{\frac { \left( b \left( \mu_{{2}}-1 \right) ^{2}{a}^{2}+
	 		b{\mu_{{2}}}^{2}\nu_{{2}} \right) }{{\nu_{{1}}}^{2}\nu_{{2}}}}
	 \esp A_2={\frac { \left( {b}^{2}a{\mu_{{2}}}^{3}-2\,{b}^{2}a{\mu_{{2}}}^{2}+{b}
	 		^{2}a\mu_{{2}}+a\nu_{{2}} \right) }{{\nu_{{1}}}^{2}{
	 			\mu_{{2}}}^{2}\nu_{{2}}}}
	  \]and one can see easily $A_1=A_2=0$ if and only if $a=b=0$ and hence  $\xi$ is biharmonic if and only if $\xi$ is biharmonic.
	 
	 We suppose now that $\mu_1<1$ and $\mu_2<1$. So $\det(M_\xi(\B_0))=0$ if and only if $\al^2=\be^2$. According to Proposition \ref{test}, if $\al^2\not=\be^2$ then $\xi$ is biharmonic if and only if it is harmonic. We have also that if $\al=\be=0$ then $\tau_2(\xi)=0$ if and only if $a=b=0$
	 
	 Suppose that $\al^2=\be^2$ and $\al\not=0$. If $a=b=0$ then $\tau_2(\xi)=0$ and $\xi$ is biharmonic not harmonic. Suppose $(a,b)\not=0$. Then the rank of $M_\xi(B_0)$ is equal to 2 and its kernel has dimension one and
	 $$v=a \left( \mu_{{2}}-1 \right) X_1-b \left( 
	 \mu_{{2}}-1 \right) \mu_{{2}}X_2+\,\mu_{{2}}X_3$$ is a generator of the kernel of $M_\xi(\B_0)$. But if $\xi$ is biharmonic then $\tau(\xi)$ is in the kernel of $M_\xi(\B_0)$ and hence it is a multiple of $v$. Recall that
	 \[\tau(\xi)=-\frac{\mu_2b}{\nu_1}X_1+\frac{ a}{\mu_2\nu_1}X_2+
	 \frac{(\mu_2-1)(\e\al^2\nu_1(\mu_1-1)+ab\mu_1)}{\mu_1\nu_1\nu_2}X_3\esp\e=\pm1.  \]But $(v,\tau(\xi))$ are linearly dependent if and only if
	 \[ \begin{cases}\di{\frac { \left( \mu_{{2}}-1 \right) \nu_{{2}}\mu_{{1}}\, \left( 
	 		-{b}^{2}{\mu_{{2}}}^{3}+{a}^{2} \right) }{\mu_{{2}}}}=0,\\
	 -b \left( \mu_{{2}}-1 \right) ^{2}\mu_{{2}}\epsilon\, \left( \mu_{{1}}
	 -1 \right) \nu_{{1}}{\alpha}^{2}-a\mu_{{1}} \left( {b}^{2}{\mu_{{2}}}^
	 {3}-2\,{b}^{2}{\mu_{{2}}}^{2}+{b}^{2}\mu_{{2}}+\nu_{{2}} \right) 
	 =0,\\
	 b\mu_{{1}} \left( \mu_{{2}}-1 \right) ^{2}{a}^{2}+{\alpha}^{2}\epsilon
	 \,\nu_{{1}} \left( \mu_{{2}}-1 \right) ^{2} \left( \mu_{{1}}-1
	 \right) a+b\mu_{{1}}{\mu_{{2}}}^{2}\nu_{{2}}
	 =0.
	  \end{cases} \]Since $(a,b)\not=(0,0)$, this is equivalent to
	  \[ \begin{cases}\di a^2={b}^{2}{\mu_{{2}}}^{3},\\
	  \di \alpha^2=-{\frac {\mu_{{1}}a \left( {b}^{2}{\mu_{{2}}}(\mu_2-1)^2+\nu_{{2}} \right) }{b \left( \mu_{{2
	  			}}-1 \right) ^{2}\mu_{{2}}\epsilon\, \left( \mu_{{1}}-1 \right) \nu_{{
	  			1}}}}=-{\frac {\mu_{{1}}b \left( {\mu_{{2}}}^{2}\nu_{{2}}+{a}^{2}(\mu_2-1)^2
	  	 \right) }{\epsilon\,\nu_{{
	  			1}} \left( \mu_{{2}}-1 \right) ^{2} \left( \mu_{{1}}-1 \right) a}}
\end{cases}	  \]and this is equivalent to
	\[ \begin{cases}\di a^2={b}^{2}{\mu_{{2}}}^{3},\\
	\di \alpha^2=-{\frac {\mu_{{1}}b \left( {\mu_{{2}}}^{2}\nu_{{2}}+{a}^{2}(\mu_2-1)^2
		\right) }{\epsilon\,\nu_{{
				1}} \left( \mu_{{2}}-1 \right) ^{2} \left( \mu_{{1}}-1 \right) a}}.
\end{cases}	  \] So $a=\e b\mu_2\sqrt{\mu_2}$ and we get the desired result.

The case of $\xi=\left( \begin{matrix}
\al&\be&a\\\be&-\al&b\\0&0&-1
\end{matrix} \right)$ can be treated identically.
			\end{proof}

			\section{Harmonic and biharmonic homomorphisms on $\mathrm{Sol}$}\label{section5}
			\begin{theo}\label{theohsol} An homomorphism of $\s$ is harmonic if and only if it is conjugate to $\xi:(\s,\prs_1)\too(\s,\prs_2)$ where:
				
			\begin{enumerate}	
				
			\item  $\prs_i=\mathrm{Diag}(1,1,\nu_i)$, $i=1,2$ and $\nu_i>0$ and either
			
			  $$[\xi=\xi_1,\;(a=b=0)\ou (\ga=0,a^2=b^2)],
			 [\xi=\xi_2,\;(a=b=0,\al^2=\be^2)]\ou 
			[\xi=\xi_3,\;(a=b=0,\al^2=\be^2)].$$

	\item  $\prs_1=\left( \begin{matrix}
	1&0&0\\0&1&0\\0&0&\nu_1
	\end{matrix} \right)$, $\prs_2=\left( \begin{matrix}
	1&1&0\\1&\mu_2&0\\0&0&\nu_2
	\end{matrix} \right)$ and $\nu_i>0$, $\mu_2>1$ and either
	\begin{eqnarray*}&&\left[\xi=\xi_1,(a=b=0)\ou \left(\ga=0,\mu_2=\frac{a^2}{b^2}\right)\right],
		 \left[\xi=\xi_2,(a=b=\al=\be=0)\ou \left(a=b=0,\mu_2=\frac{\al^2}{\be^2}\right)\right]\ou\\
		&&\left[\xi=\xi_3,(a=b=\al=\be=0)\ou \left(a=b=0,\mu_2=\frac{\al^2}{\be^2}\right)\right].\end{eqnarray*}

	\item  $\prs_1=\left( \begin{matrix}
	1&1&0\\1&\mu_1&0\\0&0&\nu_1
	\end{matrix} \right)$, $\prs_2=\left( \begin{matrix}
	1&0&0\\0&1&0\\0&0&\nu_2
	\end{matrix} \right)$ and $\nu_i>0$, $\mu_1>1$, 
	
	\begin{eqnarray*}&& \left[\xi=\xi_1,(a=b=0)\ou (\ga=0,a^2=b^2)\right],
	 \left[\xi=\xi_2,(a=b=\al=\be=0)\ou \left(a=b=0,\mu_1=\frac{\be^2}{\al^2}\right)\right],\ou\\ 
	&&	\left[\xi=\xi_3,(a=b=\al=\be=0)\ou \left(a=b=0,\mu_1=\frac{\be^2}{\al^2}\right)\right].\end{eqnarray*}

		\item  $\prs_1=\left( \begin{matrix}
		1&1&0\\1&\mu_1&0\\0&0&\nu_1
		\end{matrix} \right)$, $\prs_2=\left( \begin{matrix}
		1&1&0\\1&\mu_2&0\\0&0&\nu_2
		\end{matrix} \right)$ and $\nu_i>0$, $\mu_i>1$ and either
		
		\begin{eqnarray*}&&  \left[\xi=\xi_1,(a=b=0)\ou \left(\ga=0,\mu_2=\frac{a^2}{b^2}\right)\right],\\
		 &&\left[\xi=\xi_2,(a=b=\al=\be=0)\ou\left(a=b=0,(\al,\be)\not=(0,0),
		 \al^2\mu_1=\be^2\mu_2\right)\right]\ou\\
		&&\left[\xi=\xi_3,(a=b=\al=\be=0)\ou\left(a=b=0,
		(\al,\be)\not=(0,0),\al^2\mu_1\mu_2=\be^2\right)\right].\end{eqnarray*}The homomorphisms $\xi_i,i=1..3$ are given by
	\[ \xi_1=\left( \begin{matrix}
	0&0&a\\0&0&b\\0&0&\ga
	\end{matrix} \right),\;\xi_2=\left( \begin{matrix}
	\al&0&a\\0&\be&b\\0&0&1
	\end{matrix} \right)\esp \xi_3=\left( \begin{matrix}
	0&\be&a\\\al&0&b\\0&0&-1
	\end{matrix} \right). \]
			
			\end{enumerate}
			
	\end{theo}
	
	\begin{proof} We use Table \ref{1} to get all the conjugation classes of homomorphisms of $\s$ and for each one we compute $\tau(\xi)$.

	$\bullet$ $\prs_i=\mathrm{Diag}(1,1,\nu_i)$, $i=1,2$ and $\nu_i>0$.	
	We have
	\[ \begin{cases}\di\tau(\xi_1)=-{\frac {\gamma\,a}{\nu_{{1}}}}X_1+{\frac {
			\gamma\,b}{\nu_{{1}}}}X_2+{\frac {{a}^{2}-{b}^{2}}{\nu_{{2}}\nu_{{1}}}}X_3,\\
	\di\tau(\xi_2)=-{\frac {a}{\nu_{{1}}}}X_1+{\frac {b}{\nu_{{1
				}}}}X_2+{\frac { \left( {\alpha}^{2}-{\beta}^{2} \right) \nu_{{1}}+{a}^{2
			}-{b}^{2}}{\nu_{{2}}\nu_{{1}}}}X_3,\\
	\di\tau(\xi_3)={\frac {a}{\nu_{{1}}}}X_1-{\frac {b}{\nu_{{1
				}}}}X_2+{\frac { \left( -{\alpha}^{2}+{\beta}^{2} \right) \nu_{{1}}+{a}^{
				2}-{b}^{2}}{\nu_{{2}}\nu_{{1}}}}X_3.
		\end{cases} \]
		
$\bullet$	$\prs_1=\left( \begin{matrix}
	1&0&0\\0&1&0\\0&0&\nu_1
	\end{matrix} \right)$, $\prs_2=\left( \begin{matrix}
	1&1&0\\1&\mu_2&0\\0&0&\nu_2
	\end{matrix} \right)$ and $\nu_i>0$, $\mu_2>1$. We have	
	$$\begin{cases}\di\tau(\xi_1)=-{\frac { \left(  \left( a+2\,b \right) 
				\mu_{{2}}+a \right) \gamma}{ \left( \mu_{{2}}-1 \right) \nu_{{1}}}}X_1+{
			\frac {\gamma\, \left( b\mu_{{2}}+2\,a+b \right) }{ \left( \mu_{{2}}-1
				\right) \nu_{{1}}}}X_2+{\frac {-{b}^{2}\mu_{{2}}+{a}^{2}}{\nu_{{2}}\nu_{
					{1}}}}X_3,\\
		\di\tau(\xi_2)=-{\frac { \left( a+2\,b \right) \mu_{{2}}+
				a}{ \left( \mu_{{2}}-1 \right) \nu_{{1}}}}X_1+{\frac {b\mu_{{2}}+2\,a+b}{
				\left( \mu_{{2}}-1 \right) \nu_{{1}}}}X_2+{\frac { \left( -{\beta}^{2}
				\mu_{{2}}+{\alpha}^{2} \right) \nu_{{1}}-{b}^{2}\mu_{{2}}+{a}^{2}}{\nu
				_{{2}}\nu_{{1}}}}X_3,\\
		\\
		\di\tau(\xi_3)={\frac { \left( a+2\,b \right) \mu_{{2}}+a
			}{ \left( \mu_{{2}}-1 \right) \nu_{{1}}}}X_1-{\frac {\mu_{{2}}b+2\,a+b}{
			\left( \mu_{{2}}-1 \right) \nu_{{1}}}}X_2+{\frac { \left( -{\alpha}^{2}
			\mu_{{2}}+{\beta}^{2} \right) \nu_{{1}}-{b}^{2}\mu_{{2}}+{a}^{2}}{\nu_
			{{2}}\nu_{{1}}}}X_3.
	\end{cases}$$
	
	$\bullet$  $\prs_1=\left( \begin{matrix}
	1&1&0\\1&\mu_1&0\\0&0&\nu_1
	\end{matrix} \right)$, $\prs_2=\left( \begin{matrix}
	1&0&0\\0&1&0\\0&0&\nu_2
	\end{matrix} \right)$ and $\nu_i>0$, $\mu_1>1$. We have
	\[ \begin{cases}\di\tau(\xi_1)=-{\frac {\gamma\,a}{\nu_{{1}}}}X_1+{\frac {
			\gamma\,b}{\nu_{{1}}}}X_2+{\frac {{a}^{2}-{b}^{2}}{\nu_{{2}}\nu_{{1}}}}X_3,\\
	\di\tau(\xi_2)= -{\frac {a}{\nu_{{1}}}}X_1+{\frac {b}{\nu_{{1
				}}}}X_2+{\frac { \left( {\alpha}^{2}\nu_{{1}}+{a}^{2}-{b}^{2} \right) \mu
				_{{1}}-{\beta}^{2}\nu_{{1}}-{a}^{2}+{b}^{2}}{\nu_{{2}} \left( \mu_{{1}
				}-1 \right) \nu_{{1}}}}X_3,
		\\
	\di\tau(\xi_3)={\frac {a}{\nu_{{1}}}}X_1-{\frac {b}{\nu_{{1
				}}}}X_2+{\frac { \left( -{\alpha}^{2}\nu_{{1}}+{a}^{2}-{b}^{2} \right) 
				\mu_{{1}}+{\beta}^{2}\nu_{{1}}-{a}^{2}+{b}^{2}}{\nu_{{2}} \left( \mu_{
					{1}}-1 \right) \nu_{{1}}}}X_3.
			\end{cases} \] 
	$\bullet$ $\prs_1=\left( \begin{matrix}
	1&1&0\\1&\mu_1&0\\0&0&\nu_1
	\end{matrix} \right)$, $\prs_2=\left( \begin{matrix}
	1&1&0\\1&\mu_2&0\\0&0&\nu_2
	\end{matrix} \right)$ and $\nu_i>0$, $\mu_i>1$. We have
	\[ \begin{cases}\di\tau(\xi_1)=-{\frac { \left(  \left( a+2\,b \right) 
			\mu_{{2}}+a \right) \gamma}{ \left( \mu_{{2}}-1 \right) \nu_{{1}}}}X_1+
	{\frac {\gamma\, \left( b\mu_{{2}}+2\,a+b \right) }{ \left( \mu_{{2}}-1
			\right) \nu_{{1}}}}X_2+{\frac {-{b}^{2}\mu_{{2}}+{a}^{2}}{\nu_{{2}}\nu_{
				{1}}}}X_3,\\
	\di\tau(\xi_2)=-{\frac { \left( a+2\,b \right) \mu_{{2}}+
			a}{ \left( \mu_{{2}}-1 \right) \nu_{{1}}}}X_1+{\frac {b\mu_{{2}}+2\,a+b}{
			\left( \mu_{{2}}-1 \right) \nu_{{1}}}}X_2+{\frac { \left( {\alpha}^{2}
			\nu_{{1}}-{b}^{2}\mu_{{2}}+{a}^{2} \right) \mu_{{1}}+ \left( -{\beta}^
			{2}\nu_{{1}}+{b}^{2} \right) \mu_{{2}}-{a}^{2}}{\nu_{{2}} \left( \mu_{
				{1}}-1 \right) \nu_{{1}}}}X_3,
	\\
	\di\tau(\xi_3)={\frac { \left( a+2\,b \right) \mu_{{2}}+a
		}{ \left( \mu_{{2}}-1 \right) \nu_{{1}}}}X_1-{\frac {b\mu_{{2}}+2\,a+b}{
		\left( \mu_{{2}}-1 \right) \nu_{{1}}}}X_2+{\frac { \left(  \left( -{
			\alpha}^{2}\nu_{{1}}-{b}^{2} \right) \mu_{{2}}+{a}^{2} \right) \mu_{{1
			}}+{b}^{2}\mu_{{2}}+{\beta}^{2}\nu_{{1}}-{a}^{2}}{\nu_{{2}} \left( \mu
			_{{1}}-1 \right) \nu_{{1}}}}X_3.
		\end{cases} \]One can check that $\tau(\xi_i)=0$ are equivalent to the conditions given in the theorem.
			\end{proof}

			\begin{theo}\label{theosolbi} An homomorphism of $\s$ is biharmonic if and only if it is harmonic.
				
				\end{theo}
				
				\begin{proof} As above, we put
					\[ \xi_1=\left( \begin{matrix}
					0&0&a\\0&0&b\\0&0&\ga
					\end{matrix} \right),\;\xi_2=\left( \begin{matrix}
					\al&0&a\\0&\be&b\\0&0&1
					\end{matrix} \right)\esp \xi_3=\left( \begin{matrix}
					0&\be&a\\\al&0&b\\0&0&-1
					\end{matrix} \right). \]
		Let $\xi:(\s,\prs_1)\too(\s,\prs_2)$ an homomorphism. Table \ref{1} gives all the possible conjugation classes of $\xi$ and we will show that for each case $\xi$ is biharmonic if and only if $\xi$ is harmonic.
		
		$\bullet$ $\xi=\xi_1$ and $\prs_i=\mathrm{Diag}(1,1,\nu_i)$, $i=1,2$ and $\nu_i>0$. We have
		\[ \tau_2(\xi)=-2\,{\frac { \left( 1/2\,{\gamma}^{2}\nu_{
					{2}}+{a}^{2}-{b}^{2} \right) a\gamma}{{\nu_{{1}}}^{2}\nu_{{2}}}}X_1-2\,{
			\frac { \left( -1/2\,{\gamma}^{2}\nu_{{2}}+{a}^{2}-{b}^{2} \right) 
				\gamma\,b}{{\nu_{{1}}}^{2}\nu_{{2}}}}X_2+{\frac {{\gamma}^{2} \left( {a}^
				{2}-{b}^{2} \right) \nu_{{2}}+2\,{a}^{4}-2\,{b}^{4}}{{\nu_{{2}}}^{2}{
					\nu_{{1}}}^{2}}}X_3. \]
		One can see easily that $\tau_2(\xi)=0$ if and only if $(a=b=0)$ or $(\ga=0,a^2=b^2)$ which, according to Theorem \ref{theohsol}, is equivalent to $\xi$ is harmonic.
		
		$\bullet$ $\xi=\xi_2$ and $\prs_i=\mathrm{Diag}(1,1,\nu_i)$, $i=1,2$ and $\nu_i>0$. We have
		\[\begin{cases}\di \tau_2(\xi)=-2\,{\frac {a \left(  \left( {\alpha}^{2}-
				{\beta}^{2} \right) \nu_{{1}}+{a}^{2}-{b}^{2}+1/2\,\nu_{{2}} \right) 
			}{{\nu_{{1}}}^{2}\nu_{{2}}}}X_1-2\,{\frac { \left(  \left( {\alpha}^{2}-
			{\beta}^{2} \right) \nu_{{1}}+{a}^{2}-{b}^{2}-1/2\,\nu_{{2}} \right) b
		}{{\nu_{{1}}}^{2}\nu_{{2}}}}X_2\\\di+{\frac {2\,{\alpha}^{4}{\nu_{{1}}}^{2}-2
		\,{\beta}^{4}{\nu_{{1}}}^{2}+4\,{a}^{2}{\alpha}^{2}\nu_{{1}}-4\,{b}^{2
		}{\beta}^{2}\nu_{{1}}+2\,{a}^{4}-2\,{b}^{4}+{a}^{2}\nu_{{2}}-{b}^{2}
		\nu_{{2}}}{{\nu_{{1}}}^{2}{\nu_{{2}}}^{2}}}X_3.\end{cases}
 \]We have also
 \[ M_\xi(\B_0)=\left[ \begin {array}{ccc} {\nu_{{1}}}^{-1}&0&-2\,{\frac {a}{\nu_{{1}
 		}}}\\ \noalign{\medskip}0&{\nu_{{1}}}^{-1}&-2\,{\frac {b}{\nu_{{1}}}}
 		\\ \noalign{\medskip}-{\frac {a}{\nu_{{1}}}}&-{\frac {b}{\nu_{{1}}}}&{
 			\frac {2\,{\alpha}^{2}\nu_{{1}}+2\,{\beta}^{2}\nu_{{1}}+2\,{a}^{2}+2\,
 				{b}^{2}}{\nu_{{1}}}}\end {array} \right]\esp \det(M_\xi(\B_0))=
 		2\,{\frac {{\alpha}^{2}+{\beta}^{2}}{{\nu_{{1}}}^{2}}}. 
 		\]According to Proposition \ref{test}, if $(\al,\be)\not=(0,0)$ then $\xi$ is biharmonic if and only if it is harmonic. If $\al=\be=0$ then
 		$\xi=\xi_1$ with $\ga=1$ and we can use the arguments used in the precedent case to conclude.
		
\		
		$\bullet$ $\xi=\xi_3$ and $\prs_i=\mathrm{Diag}(1,1,\nu_i)$, $i=1,2$ and $\nu_i>0$. We have
		\[ \begin{cases}\di\tau_2(\xi)={\frac {a \left( -2\,{\alpha}^{2}\nu_{{1}}
				+2\,{\beta}^{2}\nu_{{1}}+2\,{a}^{2}-2\,{b}^{2}+\nu_{{2}} \right) }{{
					\nu_{{1}}}^{2}\nu_{{2}}}}X_1+{\frac {b \left( -2\,{\alpha}^{2}\nu_{{1}}+2
				\,{\beta}^{2}\nu_{{1}}+2\,{a}^{2}-2\,{b}^{2}-\nu_{{2}} \right) }{{\nu_
					{{1}}}^{2}\nu_{{2}}}}X_2\\\di+{\frac {-2\,{\alpha}^{4}{\nu_{{1}}}^{2}+2\,{
					\beta}^{4}{\nu_{{1}}}^{2}+4\,{a}^{2}{\beta}^{2}\nu_{{1}}-4\,{\alpha}^{
					2}{b}^{2}\nu_{{1}}+2\,{a}^{4}-2\,{b}^{4}+{a}^{2}\nu_{{2}}-{b}^{2}\nu_{
					{2}}}{{\nu_{{1}}}^{2}{\nu_{{2}}}^{2}}}X_3
			\end{cases} \]and
			\[ M_\xi(\B_0)=\left[ \begin {array}{ccc} {\nu_{{1}}}^{-1}&0&2\,{\frac {a}{\nu_{{1}}
				}}\\ \noalign{\medskip}0&{\nu_{{1}}}^{-1}&2\,{\frac {b}{\nu_{{1}}}}
				\\ \noalign{\medskip}{\frac {a}{\nu_{{1}}}}&{\frac {b}{\nu_{{1}}}}&{
					\frac {2\,{\alpha}^{2}\nu_{{1}}+2\,{\beta}^{2}\nu_{{1}}+2\,{a}^{2}+2\,
						{b}^{2}}{\nu_{{1}}}}\end {array} \right] \esp \det(M_\xi(\B_0))=
				2\,{\frac {{\alpha}^{2}+{\beta}^{2}}{{\nu_{{1}}}^{2}}}
				 \]and the situation is similar to the precedent cases.

		$\bullet$ $\xi=\xi_1$ and	$\prs_1=\left( \begin{matrix}
		1&0&0\\0&1&0\\0&0&\nu_1
		\end{matrix} \right)$, $\prs_2=\left( \begin{matrix}
		1&1&0\\1&\mu_2&0\\0&0&\nu_2
		\end{matrix} \right)$ and $\nu_i>0$, $\mu_2>1$.
		\[ \begin{cases}\di\tau(\xi)=-{\frac {\gamma\, \left(  \left( a+2\,b
				\right) \mu_{{2}}+a \right) }{ \left( \mu_{{2}}-1 \right) \nu_{{1}}}}
		X_1+{\frac {\gamma\, \left( b\mu_{{2}}+2\,a+b \right) }{ \left( \mu_{{2}}
				-1 \right) \nu_{{1}}}}X_2+{\frac {-{b}^{2}\mu_{{2}}+{a}^{2}}{\nu_{{2}}\nu
				_{{1}}}}X_3,\\
		\di\tau_2(\xi)= -2\,{\frac { \left( -{b}^{2} \left( a-b
				\right) {\mu_{{2}}}^{3}+ \left( {a}^{3}-{a}^{2}b+ \left( 1/2\,{\gamma
				}^{2}\nu_{{2}}+{b}^{2} \right) a+2\,b{\gamma}^{2}\nu_{{2}}-{b}^{3}
				\right) {\mu_{{2}}}^{2}+ \left( 3\,a{\gamma}^{2}\nu_{{2}}+2\,b{\gamma
				}^{2}\nu_{{2}}-{a}^{3}+{a}^{2}b \right) \mu_{{2}}+1/2\,a{\gamma}^{2}
				\nu_{{2}} \right) \gamma}{{\nu_{{1}}}^{2}\nu_{{2}} \left( \mu_{{2}}-1
				\right) ^{2}}}X_1\\\di+2\,{\frac {\gamma\, \left( {b}^{3}{\mu_{{2}}}^{3}-b
				\left( -1/2\,{\gamma}^{2}\nu_{{2}}+{a}^{2}+ab+{b}^{2} \right) {\mu_{{
							2}}}^{2}+ \left( a{b}^{2}+ \left( 3\,{\gamma}^{2}\nu_{{2}}+{a}^{2}
				\right) b+2\,a{\gamma}^{2}\nu_{{2}}+{a}^{3} \right) \mu_{{2}}+2\,a{
					\gamma}^{2}\nu_{{2}}+1/2\,b{\gamma}^{2}\nu_{{2}}-{a}^{3} \right) }{{
					\nu_{{1}}}^{2}\nu_{{2}} \left( \mu_{{2}}-1 \right) ^{2}}}X_2\\\di+2\,{\frac {
				\left( -{b}^{2}\mu_{{2}}+{a}^{2} \right)  \left( {b}^{2}{\mu_{{2}}}^{
					2}+ \left( 1/2\,{\gamma}^{2}\nu_{{2}}+{a}^{2}-{b}^{2} \right) \mu_{{2}
				}+3/2\,{\gamma}^{2}\nu_{{2}}-{a}^{2} \right) }{{\nu_{{1}}}^{2}{\nu_{{2
					}}}^{2} \left( \mu_{{2}}-1 \right) }}X_3
			
			\end{cases} \]
			
			Suppose that $\xi$ is biharmonic not harmonic. Then, by virtue of Theorem \ref{theohsol}, $(a,b)\not=0$ and ($\ga\not=0$ or $\mu_2\not=\frac{a^2}{b^2}$). If $\mu_2=\frac{a^2}{b^2}$ then a direct computation shows that
			\[ \tau_2(\xi)=-{\frac { \left( a+b \right) ^{2}{\gamma}^
					{3}a}{ \left( a-b \right) ^{2}{\nu_{{1}}}^{2}}}X_1+{\frac {b \left( a+b
					\right) ^{2}{\gamma}^{3}}{ \left( a-b \right) ^{2}{\nu_{{1}}}^{2}}}X_2 \]
			and since $(a,b)\not=(0,0)$, $\ga\not=0$ and $\mu_2>1$ this is impossible so we must have $\mu_2\not=\frac{a^2}{b^2}$. In this case, since the last coordinate of $\tau_2(\xi)$ vanishes, we get
			\[ \left( \frac12\,{\gamma}^{2}\nu_{{2}}+{a}^{2} \right) \mu_{{2}}+\frac32\,{
				\gamma}^{2}\nu_{{2}}-{a}^{2}+ \left( {\mu_{{2}}}^{2}-\mu_{{2}}
			\right) {b}^{2}= \frac12\,{\gamma}^{2}\nu_{{2}}\mu_{{2}}+{a}^{2} (\mu_2-1) +\frac32\,{
				\gamma}^{2}\nu_{{2}}+ \left( {\mu_{{2}}}^{2}-\mu_{{2}}
			\right) {b}^{2}=0
			\]
			But since $\mu_2>1$, this is equivalent to $\ga=a=b=0$ which is a contradiction.
			 Finally, $\xi$ is biharmonic if and only if it is harmonic.\\

		$\bullet$ $\xi=\xi_2$ and	$\prs_1=\left( \begin{matrix}
		1&0&0\\0&1&0\\0&0&\nu_1
		\end{matrix} \right)$, $\prs_2=\left( \begin{matrix}
		1&1&0\\1&\mu_2&0\\0&0&\nu_2
		\end{matrix} \right)$ and $\nu_i>0$, $\mu_2>1$.	 We have
		\[ M_\xi(\B_0)=\left[ \begin {array}{ccc} {\nu_{{1}}}^{-1}&-{\nu_{{1}}}^{-1}&-2\,{
			\frac {a}{\nu_{{1}}}}\\ \noalign{\medskip}-{\nu_{{1}}}^{-1}&{\frac {
				\mu_{{2}}}{\nu_{{1}}}}&-2\,{\frac {b\mu_{{2}}}{\nu_{{1}}}}
		\\ \noalign{\medskip}{\frac {-a+b}{\nu_{{1}}}}&{\frac {-b\mu_{{2}}+a}{
				\nu_{{1}}}}&{\frac {2\,\nu_{{1}}{\beta}^{2}\mu_{{2}}+2\,{\alpha}^{2}
				\nu_{{1}}+2\,{b}^{2}\mu_{{2}}+2\,{a}^{2}}{\nu_{{1}}}}\end {array}
		\right] \esp\det(M_\xi(\B_0))=2\,{\frac { \left( \mu_{{2}}-1 \right)  \left( {\beta}^{2}\mu_{{2}}+{
					\alpha}^{2} \right) }{{\nu_{{1}}}^{2}}}
		 \]If $(\al,\be)\not=(0,0)$ then, according to Proposition \ref{test}, $\xi$ is biharmonic if and only if it is harmonic. If $\al=\be=0$ then
		 $\xi=\xi_1$ with $\ga=1$ and we can use the arguments used in the precedent case to conclude.

	$\bullet$ $\xi=\xi_3$ and	$\prs_1=\left( \begin{matrix}
	1&0&0\\0&1&0\\0&0&\nu_1
	\end{matrix} \right)$, $\prs_2=\left( \begin{matrix}
	1&1&0\\1&\mu_2&0\\0&0&\nu_2
	\end{matrix} \right)$ and $\nu_i>0$, $\mu_2>1$.	We have
	\[ M_\xi(\B_0)=\left[ \begin {array}{ccc} {\nu_{{1}}}^{-1}&-{\nu_{{1}}}^{-1}&2\,{
		\frac {a}{\nu_{{1}}}}\\ \noalign{\medskip}-{\nu_{{1}}}^{-1}&{\frac {
			\mu_{{2}}}{\nu_{{1}}}}&2\,{\frac {b\mu_{{2}}}{\nu_{{1}}}}
	\\ \noalign{\medskip}{\frac {a-b}{\nu_{{1}}}}&{\frac {b\mu_{{2}}-a}{
			\nu_{{1}}}}&{\frac {2\,{\alpha}^{2}\mu_{{2}}\nu_{{1}}+2\,{b}^{2}\mu_{{
					2}}+2\,{\beta}^{2}\nu_{{1}}+2\,{a}^{2}}{\nu_{{1}}}}\end {array}
	\right] \esp\det(M_\xi(\B_0))=2\,{\frac { \left( \mu_{{2}}-1 \right)  \left( {\alpha}^{2}\mu_{{2}}+{
				\beta}^{2} \right) }{{\nu_{{1}}}^{2}}}
	 \]If $(\al,\be)\not=(0,0)$ then, according to Proposition \ref{test}, $\xi$ is biharmonic if and only if it is harmonic. If $\al=\be=0$ then
	 $\xi=\xi_1$ with $\ga=1$ and we can use the arguments used in the precedent case to conclude.

	$\bullet$ 	$\xi=\xi_1$ and $\prs_1=\left( \begin{matrix}
	1&1&0\\1&\mu_1&0\\0&0&\nu_1
	\end{matrix} \right)$, $\prs_2=\left( \begin{matrix}
	1&0&0\\0&1&0\\0&0&\nu_2
	\end{matrix} \right)$ and $\nu_i>0$, $\mu_1>1$. We have
	\[ \tau_2(\xi)=-2\,{\frac {\gamma\, \left( 1/2\,{\gamma}^
			{2}\nu_{{2}}+{a}^{2}-{b}^{2} \right) a}{{\nu_{{1}}}^{2}\nu_{{2}}}}X_1-2
	\,{\frac {b \left( -1/2\,{\gamma}^{2}\nu_{{2}}+{a}^{2}-{b}^{2}
			\right) \gamma}{{\nu_{{1}}}^{2}\nu_{{2}}}}X_2+{\frac {{\gamma}^{2}
			\left( {a}^{2}-{b}^{2} \right) \nu_{{2}}+2\,{a}^{4}-2\,{b}^{4}}{{\nu_
				{{2}}}^{2}{\nu_{{1}}}^{2}}}X_3 \]
	One can see easily that $\tau_2(\xi)=0$ if and only if $(a=b=0)$ or $(\ga=0,a^2=b^2)$ which, according to Theorem \ref{theohsol}, is equivalent to $\xi$ is harmonic.

	$\bullet$ 	$\xi=\xi_2$ and $\prs_1=\left( \begin{matrix}
	1&1&0\\1&\mu_1&0\\0&0&\nu_1
	\end{matrix} \right)$, $\prs_2=\left( \begin{matrix}
	1&0&0\\0&1&0\\0&0&\nu_2
	\end{matrix} \right)$ and $\nu_i>0$, $\mu_1>1$. We have
	\[ M_\xi(\B_0)=\left[ \begin {array}{ccc} {\nu_{{1}}}^{-1}&0&-2\,{\frac {a}{\nu_{{1}
			}}}\\ \noalign{\medskip}0&{\nu_{{1}}}^{-1}&-2\,{\frac {b}{\nu_{{1}}}}
			\\ \noalign{\medskip}-{\frac {a}{\nu_{{1}}}}&-{\frac {b}{\nu_{{1}}}}&{
				\frac { \left( 2\,{\alpha}^{2}\nu_{{1}}+2\,{a}^{2}+2\,{b}^{2} \right) 
					\mu_{{1}}+2\,{\beta}^{2}\nu_{{1}}-2\,{a}^{2}-2\,{b}^{2}}{ \left( \mu_{
						{1}}-1 \right) \nu_{{1}}}}\end {array} \right] 
			\esp\det(M_\xi(\B_0))=2\,{\frac {{\alpha}^{2}\mu_{{1}}+{\beta}^{2}}{{\nu_{{1}}}^{2} \left( 
					\mu_{{1}}-1 \right) }}. \]
			If $(\al,\be)\not=(0,0)$ then, according to Proposition \ref{test}, $\xi$ is biharmonic if and only if it is harmonic. If $\al=\be=0$ then
			$\xi=\xi_1$ with $\ga=1$ and we can use the arguments used in the precedent case to conclude.

	$\bullet$ 	$\xi=\xi_3$ and $\prs_1=\left( \begin{matrix}
	1&1&0\\1&\mu_1&0\\0&0&\nu_1
	\end{matrix} \right)$, $\prs_2=\left( \begin{matrix}
	1&0&0\\0&1&0\\0&0&\nu_2
	\end{matrix} \right)$ and $\nu_i>0$, $\mu_1>1$. We have
	\[ M_\xi(\B_0)=\left[ \begin {array}{ccc} {\nu_{{1}}}^{-1}&0&2\,{\frac {a}{\nu_{{1}}
		}}\\ \noalign{\medskip}0&{\nu_{{1}}}^{-1}&2\,{\frac {b}{\nu_{{1}}}}
		\\ \noalign{\medskip}{\frac {a}{\nu_{{1}}}}&{\frac {b}{\nu_{{1}}}}&{
			\frac { \left( 2\,{\alpha}^{2}\nu_{{1}}+2\,{a}^{2}+2\,{b}^{2} \right) 
				\mu_{{1}}+2\,{\beta}^{2}\nu_{{1}}-2\,{a}^{2}-2\,{b}^{2}}{ \left( \mu_{
					{1}}-1 \right) \nu_{{1}}}}\end {array} \right] \esp\det(M_\xi(\B_0))=2\,{\frac {{\alpha}^{2}\mu_{{1}}+{\beta}^{2}}{{\nu_{{1}}}^{2} \left( 
				\mu_{{1}}-1 \right) }}.
		 \]The situation is similar to the precedent case.
	
	$\bullet$ $\xi=\xi_1$ and $\prs_1=\left( \begin{matrix}
	1&1&0\\1&\mu_1&0\\0&0&\nu_1
	\end{matrix} \right)$, $\prs_2=\left( \begin{matrix}
	1&1&0\\1&\mu_2&0\\0&0&\nu_2
	\end{matrix} \right)$ and $\nu_i>0$, $\mu_i>1$. We have
	\[ \begin{cases}\di\tau_2(\xi)=-2\,{\frac {\gamma\, \left( -{b}^{2}
			\left( a-b \right) {\mu_{{2}}}^{3}+ \left( {a}^{3}-{a}^{2}b+ \left( 1
			/2\,{\gamma}^{2}\nu_{{2}}+{b}^{2} \right) a+2\,b{\gamma}^{2}\nu_{{2}}-
			{b}^{3} \right) {\mu_{{2}}}^{2}+ \left( 3\,a{\gamma}^{2}\nu_{{2}}+2\,b
			{\gamma}^{2}\nu_{{2}}-{a}^{3}+{a}^{2}b \right) \mu_{{2}}+1/2\,a{\gamma
			}^{2}\nu_{{2}} \right) }{{\nu_{{1}}}^{2}\nu_{{2}} \left( \mu_{{2}}-1
			\right) ^{2}}}X_1\\+\di2\,{\frac {\gamma\, \left( {b}^{3}{\mu_{{2}}}^{3}-b
			\left( -1/2\,{\gamma}^{2}\nu_{{2}}+{a}^{2}+ab+{b}^{2} \right) {\mu_{{
						2}}}^{2}+ \left( a{b}^{2}+ \left( 3\,{\gamma}^{2}\nu_{{2}}+{a}^{2}
			\right) b+2\,a{\gamma}^{2}\nu_{{2}}+{a}^{3} \right) \mu_{{2}}+2\,a{
				\gamma}^{2}\nu_{{2}}+1/2\,b{\gamma}^{2}\nu_{{2}}-{a}^{3} \right) }{{
				\nu_{{1}}}^{2}\nu_{{2}} \left( \mu_{{2}}-1 \right) ^{2}}}X_2+\\\di2\,{\frac {
			\left( -{b}^{2}\mu_{{2}}+{a}^{2} \right)  \left( {b}^{2}{\mu_{{2}}}^{
				2}+ \left( 1/2\,{\gamma}^{2}\nu_{{2}}+{a}^{2}-{b}^{2} \right) \mu_{{2}
			}+3/2\,{\gamma}^{2}\nu_{{2}}-{a}^{2} \right) }{{\nu_{{1}}}^{2}{\nu_{{2
				}}}^{2} \left( \mu_{{2}}-1 \right) }}X_3
		\end{cases} \]One can see that $\tau_2(\xi)$ is the same as in the case
		$\xi=\xi_1$ and $\prs_1=\left( \begin{matrix}
		1&0&0\\0&1&0\\0&0&\nu_1
		\end{matrix} \right)$, $\prs_2=\left( \begin{matrix}
		1&1&0\\1&\mu_2&0\\0&0&\nu_2
		\end{matrix} \right)$ and we can use the same arguments to conclude.

	$\bullet$ $\xi=\xi_2$ and $\prs_1=\left( \begin{matrix}
	1&1&0\\1&\mu_1&0\\0&0&\nu_1
	\end{matrix} \right)$, $\prs_2=\left( \begin{matrix}
	1&1&0\\1&\mu_2&0\\0&0&\nu_2
	\end{matrix} \right)$ and $\nu_i>0$, $\mu_i>1$. We have
\[ M_\xi(\B_0)	\left[ \begin {array}{ccc} {\nu_{{1}}}^{-1}&-{\nu_{{1}}}^{-1}&-2\,{
		\frac {a}{\nu_{{1}}}}\\ \noalign{\medskip}-{\nu_{{1}}}^{-1}&{\frac {
			\mu_{{2}}}{\nu_{{1}}}}&-2\,{\frac {\mu_{{2}}b}{\nu_{{1}}}}
	\\ \noalign{\medskip}{\frac {-a+b}{\nu_{{1}}}}&{\frac {-\mu_{{2}}b+a}{
			\nu_{{1}}}}&{\frac { \left( 2\,{\alpha}^{2}\nu_{{1}}+2\,{b}^{2}\mu_{{2
				}}+2\,{a}^{2} \right) \mu_{{1}}+ \left( 2\,{\beta}^{2}\nu_{{1}}-2\,{b}
				^{2} \right) \mu_{{2}}-2\,{a}^{2}}{\nu_{{1}} \left( \mu_{{1}}-1
				\right) }}\end {array} \right]\esp\det(M_\xi(\B_0))=2\,{\frac { \left( \mu_{{2}}-1 \right)  \left( {\alpha}^{2}\mu_{{1}}+{
					\beta}^{2}\mu_{{2}} \right) }{{\nu_{{1}}}^{2} \left( \mu_{{1}}-1
				\right) }}.
	  \]If $(\al,\be)\not=(0,0)$ then, according to Proposition \ref{test}, $\xi$ is biharmonic if and only if it is harmonic. If $\al=\be=0$ then
	  $\xi=\xi_1$ with $\ga=1$ and we can use the arguments used in the precedent case to conclude.

	$\bullet$ $\xi=\xi_3$ and $\prs_1=\left( \begin{matrix}
	1&1&0\\1&\mu_1&0\\0&0&\nu_1
	\end{matrix} \right)$, $\prs_2=\left( \begin{matrix}
	1&1&0\\1&\mu_2&0\\0&0&\nu_2
	\end{matrix} \right)$ and $\nu_i>0$, $\mu_i>1$. We have
	\[ M_\xi(\B_0)=\left[ \begin {array}{ccc} {\nu_{{1}}}^{-1}&-{\nu_{{1}}}^{-1}&2\,{
		\frac {a}{\nu_{{1}}}}\\ \noalign{\medskip}-{\nu_{{1}}}^{-1}&{\frac {
			\mu_{{2}}}{\nu_{{1}}}}&2\,{\frac {\mu_{{2}}b}{\nu_{{1}}}}
	\\ \noalign{\medskip}{\frac {a-b}{\nu_{{1}}}}&{\frac {\mu_{{2}}b-a}{
			\nu_{{1}}}}&{\frac { \left(  \left( 2\,{\alpha}^{2}\nu_{{1}}+2\,{b}^{2
			} \right) \mu_{{2}}+2\,{a}^{2} \right) \mu_{{1}}-2\,{b}^{2}\mu_{{2}}+2
			\,{\beta}^{2}\nu_{{1}}-2\,{a}^{2}}{\nu_{{1}} \left( \mu_{{1}}-1
			\right) }}\end {array} \right] \esp\det(M_\xi(B_0))=2\,{\frac { \left( \mu_{{2}}-1 \right)  \left( {\alpha}^{2}\mu_{{1}}
			\mu_{{2}}+{\beta}^{2} \right) }{{\nu_{{1}}}^{2} \left( \mu_{{1}}-1
			\right) }}.
	 \]The situation is similar to the precedent case.
					\end{proof}
\section{Harmonic and biharmonic homomorphisms of $\mathfrak{su}(2)$}\label{section6}	

The following proposition is a consequence of \cite[Proposition 2.5]{boucetta}.
\begin{pr}\label{bi} Let $\xi:(\mathfrak{su}(2),\prs_1)\too(\mathfrak{su}(2),\prs_2)$ be an automorphism. If $\prs_1$ or $\prs_2$ is bi-invariant then $\xi$ is harmonic.
	
	\end{pr}

Any homomorphism of $\mathrm{su(2)}$ is an automorphism and it is a product $\xi_3(a)\circ\xi_2(b)\circ\xi_1(c)$ where
\[ \xi_1(a)=\left(\begin{array}{ccc}1&0&0\\0&\cos(a) &\sin(a)\\
0&-\sin(a)&\cos(a) \end{array}   \right),
\xi_2(a)=\left(\begin{array}{ccc}\cos(a)&0&\sin(a)\\0&1 &0\\
-\sin(a)&0&\cos(a) \end{array}   \right),\xi_3(a)=\left(\begin{array}{ccc}\cos(a)&\sin(a)&0\\-\sin(a)& \cos(a)&0\\0&0&1 \end{array}   \right). \]
		If $\xi_i:(\mathfrak{su}(2),\prs_1)\too(\mathfrak{su}(2),\prs_2)$ with $\prs_j=\mathrm{Diag}(\la_j,\mu_j,\nu_j)$ then
			\[ \begin{cases}\di\tau(\xi_1(a))=-{\frac {\sin \left( a \right) \cos
				\left( a \right)  \left( \mu_{{2}}-\nu_{{2}} \right)  \left( \mu_{{1}
				}-\nu_{{1}} \right) }{\lambda_{{2}}\mu_{{1}}\nu_{{1}}}}X_1
		,\\
		\di\tau_2(\xi_1(a))=-2\,{\frac { \left( \mu_{{2}}-\nu_{{2}}
				\right) ^{2} \left( -\nu_{{1}}+\mu_{{1}} \right) ^{2}\cos \left( a
				\right)  \left(  \left( \cos \left( a \right)  \right) ^{2}-1/2
				\right) \sin \left( a \right) }{{\mu_{{1}}}^{2}{\nu_{{1}}}^{2}{
					\lambda_{{2}}}^{2}}}X_1,\\
			\di\tau(\xi_2(a))={\frac {\sin \left( a \right) \cos
				\left( a \right)  \left( \lambda_{{2}}-\nu_{{2}} \right)  \left( 
				\lambda_{{1}}-\nu_{{1}} \right) }{\mu_{{2}}\lambda_{{1}}\nu_{{1}}}}
		X_2,\\
		\di\tau_2(\xi_2(a))={\frac { \left( 2\, \left( \cos \left( a
				\right)  \right) ^{2}-1 \right) \cos \left( a \right)  \left( \lambda
				_{{2}}-\nu_{{2}} \right) ^{2} \left( -\nu_{{1}}+\lambda_{{1}} \right) 
				^{2}\sin \left( a \right) }{{\lambda_{{1}}}^{2}{\nu_{{1}}}^{2}{\mu_{{2
						}}}^{2}}}X_2,\\
				\di\tau(\xi_3(a))=-{\frac {\cos \left( a \right) \sin
				\left( a \right)  \left( \lambda_{{2}}-\mu_{{2}} \right)  \left( 
				\lambda_{{1}}-\mu_{{1}} \right) }{\lambda_{{1}}\mu_{{1}}\nu_{{2}}}}X_3,\\
		\di\tau_2(\xi_3(a))=-2\,{\frac {\cos \left( a \right) \sin
				\left( a \right)  \left(  \left( \cos \left( a \right)  \right) ^{2}-
				1/2 \right)  \left( \mu_{{2}}-\lambda_{{2}} \right) ^{2} \left( \mu_{{
						1}}-\lambda_{{1}} \right) ^{2}}{{\lambda_{{1}}}^{2}{\mu_{{1}}}^{2}{\nu
					_{{2}}}^{2}}}X_3.
			\end{cases} \]So we get:
			\begin{pr}\label{hsu2}\begin{enumerate}\item If $\mu_2=\nu_2$ or $\mu_1=\nu_1$ then $\xi_1(a)$ is harmonic.
			\item If $\mu_2\not=\nu_2$ and $\mu_1\not=\nu_1$ then $\xi_1(a)$ is harmonic if and only if $\sin(2a)=0$ and 
			$\xi_1(a)$ is biharmonic not harmonic if and only if $\cos(a)^2=\frac12$.		
			\item If $\la_2=\nu_2$ or $\la_1=\nu_1$ then $\xi_2(a)$ is harmonic.
			\item If $\la_2\not=\nu_2$ and $\la_1\not=\nu_1$ then $\xi_2(a)$ is harmonic if and only if $\sin(2a)=0$ and 
			$\xi_2(a)$ is biharmonic not harmonic if and only if $\cos(a)^2=\frac12$.
			\item If $\la_2=\mu_2$ or $\la_1=\mu_1$ then $\xi_3(a)$ is harmonic.
			\item If $\la_2\not=\mu_2$ and $\la_1\not=\mu_1$ then $\xi_3(a)$ is harmonic if and only if $\sin(2a)$ and 
			$\xi_3(a)$ is biharmonic not harmonic if and only if $\cos(a)^2=\frac12$.			
					\end{enumerate}
				
			\end{pr}
			
			\begin{theo}\label{theohsu2} We consider the automorphism $$\xi=\xi_3(a)\circ\xi_2(b)\circ\xi_1(c):
				\left(\mathfrak{su}(2),\mathrm{diag}(\lambda_1,\mu_1,\nu_1)\right)\too\left(\mathfrak{su}(2),\mathrm{diag}(\lambda_2,\mu_2,\nu_2)\right),\;0\leq\nu_i<\mu_i\leq\la_i,i=1,2.
				$$
				\begin{enumerate}\item If $\left(0<\nu_1<\mu_1<\la_1,0<\nu_2<\mu_2<\la_2\right)$  or $\left(0<\nu_1<\mu_1<\la_1,0<\nu_2<\mu_2<\la_2\right)$  then $\xi$
				   is harmonic if and only if one of the following condition holds:
				\begin{enumerate}\item[$(i)$] $\cos(b)=0$, $\sin(b)=1$ and $\sin(2(a-c))=0$,
					\item[$(ii)$] $\cos(b)=0$, $\sin(b)=-1$ and $\sin(2(a+c))=0$,
					\item[$(iii)$] $\sin(b)=0$ and $\sin(2c)=\sin(2a)=0$.
					
					\end{enumerate}
		\item	If $\left(0<\nu_1<\mu_1<\la_1,0<\nu_2=\mu_2<\la_2\right)$ then $\xi$
		is harmonic if and only if one of the following condition holds:
		\begin{enumerate}
			\item[$(i)$] $\cos(b)=0$, $\sin(b)=1$ and $\sin(2(a-c))=0$,
				\item[$(ii)$] $\cos(b)=0$, $\sin(b)=-1$ and $\sin(2(a+c))=0$,
				\item[$(iii)$] $\sin(b)=0$ and $\sin(a)=0$.
				\item[$(iv)$] $\sin(b)=0$, $\sin(2c)=0$ and $\cos(a)=0$.

			\end{enumerate}
		
		\item If $\left(0<\nu_1<\mu_1=\la_1,0<\nu_2<\mu_2<\la_2\right)$	 then $\xi$
		is harmonic if and only if one of the following condition holds:
		\begin{enumerate}
			\item[$(i)$] $\cos(b)=0$, $\sin(b)=1$ and $\sin(2(a-c))=0$,
			\item[$(ii)$] $\cos(b)=0$, $\sin(b)=-1$ and $\sin(2(a+c))=0$,
			\item[$(iii)$] $\sin(b)=0$ and $\sin(a)=\sin(2c)=0$.
			\item[$(iv)$] $\sin(b)=0$ and $\cos(a)=\sin(2c)=0$.

		\end{enumerate}	
					
					\item  If $\left(0<\nu_1<\mu_1<\la_1,0<\nu_2<\mu_2=\la_2\right)$ then $\xi$ is harmonic if and only if $\cos(b)=0$ or $(\sin(b)=\sin(2c)=0)$.
					
				\item   If $\left(0<\nu_1=\mu_1<\la_1,0<\nu_2<\mu_2<\la_2\right)$  then $\xi$ is harmonic if and only if $\cos(b)=0$ or $(\sin(b)=\sin(2a)=0)$.
				\item   If $\left(0<\nu_1=\mu_1<\la_1,0<\nu_2=\mu_2<\la_2\right)$  then $\xi$ is harmonic if and only if $(\cos(b)=0)$, $(\cos(a))=0$ or $(\sin(b)=\sin(a)=0)$.
				\item   If $\left(0<\nu_1=\mu_1<\la_1,0<\nu_2<\mu_2=\la_2\right)$  then $\xi$ is harmonic if and only if $\sin(2b)=0$.
				\item   If $\left(0<\nu_1<\mu_1=\la_1,0<\nu_2<\mu_2=\la_2\right)$  then $\xi$ is harmonic if and only  if $\cos(b)\cos(c)=0$ or $(\sin(b)=\sin(c)=0)$.
				
				\item   If $\left(0<\nu_1<\mu_1=\la_1,0<\nu_2=\mu_2<\la_2\right)$  then $\xi$ is harmonic if and only  if one of the following situations holds
				\begin{enumerate}\item[$(i)$] $\cos(b)=0$, $\sin(b)=1$ and $\sin(2(a-c))=0$,
					\item[$(ii)$] $\cos(b)=0$, $\sin(b)=-1$ and $\sin(2(a+c))=0$,
					\item[$(iii)$] $\cos(c)=0$ and $\sin(2a)=0$,
					\item[$(iv)$] $\sin(b)=\sin(c)=0$,
					\item[$(v)$] $\cos(a)=(-1)^k\frac{\sin(c)}{\sqrt{\sin^2(c)+\sin^2(b)\cos^2(c)}}$ and $\sin(a)=(-1)^{k+1}\frac{\sin(b)\cos(c)}{\sqrt{\sin^2(c)+\sin^2(b)\cos^2(c)}}$.
					\end{enumerate}
				\end{enumerate}
				
				\end{theo}
				
				\begin{proof}  We have{\small
						\[ \begin{cases}\di\tau(\xi)=A_1X_1+A_2X_2+A_3X_3,\\
						\di	\lambda_{{2}}\lambda_{{1}}\mu_{{1}}\nu_{{1}}A_1=\cos \left( b \right)  \left( \sin \left( a \right) \sin \left( b
						\right) \lambda_{{1}} \left( \mu_{{1}}-\nu_{{1}} \right)  \left( \cos
						\left( c \right)  \right) ^{2}-\sin \left( c \right) \lambda_{{1}}
						\cos \left( a \right)  \left( \mu_{{1}}-\nu_{{1}} \right) \cos \left( 
						c \right) +\sin \left( a \right) \sin \left( b \right) \nu_{{1}}
						\left( \lambda_{{1}}-\mu_{{1}} \right)  \right)  \left( \mu_{{2}}-\nu
						_{{2}} \right)\\
						=\cos(b)\left( \mu_{{2}}-\nu
						_{{2}} \right) R ,\\
						\di \mu_{{2}}\lambda_{{1}}\mu_{{1}}\nu_{{1}}A_2
						=\cos \left( b \right)  \left( \sin \left( b \right)  \left( \lambda_{{
								1}} \left( \mu_{{1}}-\nu_{{1}} \right)  \left( \cos \left( c \right) 
						\right) ^{2}+\nu_{{1}} \left( \lambda_{{1}}-\mu_{{1}} \right) 
						\right) \cos \left( a \right) +\cos \left( c \right) \sin \left( a
						\right) \sin \left( c \right) \lambda_{{1}} \left( \mu_{{1}}-\nu_{{1}
						} \right)  \right)  \left( \lambda_{{2}}-\nu_{{2}} \right)\\ 
						=\left( \lambda_{{2}}-\nu_{{2}} \right)\cos(b)S,\\
						\di	z:=-\nu_{{2}}\lambda_{{1}}\mu_{{1}}\nu_{{1}}A_3=
						\left( \lambda_{{2}}-\mu_{{2}} \right)  \left( 2\,\cos \left( c
						\right) \sin \left( b \right) \sin \left( c \right) \lambda_{{1}}
						\left( \mu_{{1}}-\nu_{{1}} \right)  \left( \cos \left( a \right) 
						\right) ^{2}\right.\\\left.+ \left( \lambda_{{1}} \left(  \left( \cos \left( b
						\right)  \right) ^{2}-2 \right)  \left( \mu_{{1}}-\nu_{{1}} \right) 
						\left( \cos \left( c \right)  \right) ^{2}+\nu_{{1}} \left( \lambda_{
							{1}}-\mu_{{1}} \right)  \left( \cos \left( b \right)  \right) ^{2}+
						\lambda_{{1}} \left( \mu_{{1}}-\nu_{{1}} \right)  \right) \sin \left( 
						a \right) \cos \left( a \right) -\cos \left( c \right) \sin \left( b
						\right) \sin \left( c \right) \lambda_{{1}} \left( \mu_{{1}}-\nu_{{1}
						} \right)  \right). 
						\end{cases} \]}	
					On the other hand, the following relations are straightforward to establish:
					\begin{equation}\label{smart} \begin{cases}R\cos(a)-S\sin(a)=-   \lambda_{{1}} \left( \mu_{{1}}-\nu_{{1}} \right) \sin \left( c \right) \cos \left( 
					c \right),\\
					R\sin(a)+S\cos(a)= \sin \left( b \right)   \left( \lambda_{{1}
					} \left( \mu_{{1}}-\nu_{{1}} \right)  \left( \cos \left( c \right) 
					\right) ^{2}+\nu_{{1}} \left( \lambda_{{1}}-\mu_{{1}} \right) 
					\right) 
					\end{cases} \end{equation}and if $\cos(b)=0$ then
					\begin{equation}\label{cosb} z=\begin{cases}\frac12\sin(2(c-a))\left( 
					\lambda_{{2}}-\mu_{{2}} \right) \lambda_{{1}} \left( \mu_{{1}}-\nu_{{1
						}} \right)\quad\mbox{if}\quad\sin(b)=1,\\
						\frac12\sin(2(c+a))\left( 
						\lambda_{{2}}-\mu_{{2}} \right) \lambda_{{1}} \left( \mu_{{1}}-\nu_{{1
							}} \right)\quad\mbox{if}\quad\sin(b)=-1.
							\end{cases} \end{equation}
					 Suppose that  $\left(0<\nu_1<\mu_1<\la_1,0<\nu_2<\mu_2<\la_2\right)$. Then
						$\xi$ is harmonic if and only if 
						\[ R\cos(b)=S\cos(b)=z=0. \] 
						We distinguish two cases:
						
						$\bullet$ $\cos(b)=0$. Then $\xi$ is is harmonic if and only if $z=0$
						and, by virtue of \eqref{cosb}, we get the desired result.
										
										$\bullet$ $\cos(b)\not=0$ then from \eqref{smart} $\sin(b)=0$ 	and $\sin(c)\cos(c)=0$ and one can check easily that $\xi$ is harmonic if and only if $\cos(a)\sin(a)=0$.

			Except the last case, all the other cases can be deduced in the same way. Let us complete the proof by treating the last case. We suppose that $\left(0<\nu_1<\mu_1=\la_1,0<\nu_2=\mu_2<\la_2\right)$. Then
				\[ \begin{cases}\di\tau(\xi)=
	{\frac {\cos \left( b \right) \cos
			\left( c \right)  \left( \lambda_{{1}}-\nu_{{1}} \right)    \left( \lambda_{{2}}
			-\mu_{{2}} \right) R_1 }{\mu_{{2}}\lambda_{{1}}\nu_{{1}}}}X_2-\,{\frac {2
			\left( \lambda_{{1}}-\nu_{{1}} \right)  \left( \lambda_{{2}}-\mu_{{2}
			} \right)  S_1 }{\mu_{{2}}\lambda_{{
					1}}\nu_{{1}}}}X_3,\\
	\di R_1= \sin
	\left( a \right) \sin \left( c \right) +\cos \left( a \right) \sin
	\left( b \right) \cos \left( c \right)  ,\\
	\di S_1= \sin \left( b \right)  \left( \cos \left( a \right) 
	\right) ^{2}\sin \left( c \right) \cos \left( c \right) +\frac12\,\sin
	\left( a \right)  \left( 1+ \left(  \left( \cos \left( b \right) 
	\right) ^{2}-2 \right)  \left( \cos \left( c \right)  \right) ^{2}
	\right) \cos \left( a \right) -\frac12\,\sin \left( b \right) \sin
	\left( c \right) \cos \left( c \right)  .					\end{cases} \]
	
	If $\cos(b)=0$ then 
	\[ S_1=\begin{cases}\frac14\sin(2(c-a))\quad\mbox{if}\quad\sin(b)=1,\\-\frac14\sin(2(c+a))\quad\mbox{if}\quad\sin(b)=-1\end{cases} \]
	and we get $(i)$ and $(ii)$.
	
	If $\cos(c)=0$ then $S_1=\frac14\sin(2a)$ and we get $(iii)$.
	
	If $\sin(b)=\sin(c)=0$ the $S_1=R_1=0$ and hence $\xi$ is harmonic.
	
	Suppose now that $\cos(b)\not=0$,  $\cos(c)\not=0$ and $(\sin(b),\sin(c))\not=(0,0)$. Then $\xi$ is harmonic if and only if $R_1=S_1=0$. We have
	\[ R_1=\sin(a)\frac{\sin(c)}{\sqrt{\sin^2(c)+\sin^2(b)\cos^2(c)}}+\cos(a)\frac{\sin(b)\cos(c)}{\sqrt{\sin^2(c)+\sin^2(b)\cos^2(c)}}=\sin(a+\al) \]where
	\[ \cos(\al)=\frac{\sin(c)}{\sqrt{\sin^2(c)+\sin^2(b)\cos^2(c)}}\esp \sin(\al)=
	\frac{\sin(b)\cos(c)}{\sqrt{\sin^2(c)+\sin^2(b)\cos^2(c)}}. \]
	So $R_1=0$ if and only if $a+\al=k\pi$ where $k\in\Z$. Thus
	\[ \cos(a)=(-1)^k\cos(\al)=(-1)^k\frac{\sin(c)}{\sqrt{\sin^2(c)+\sin^2(b)\cos^2(c)}}\esp \sin(a)=-(-1)^k\sin(\al)=(-1)^{k+1}\frac{\sin(b)\cos(c)}{\sqrt{\sin^2(c)+\sin^2(b)\cos^2(c)}}. \]If we replace $\cos(a)$ and $\sin(a)$ in $S_1$, we get $S_1=0$ which completes the proof.
								\end{proof}

			The situation for biharmonic homomorphisms is more complicated. We have the following non trivial biharmonic homomorphism which is not harmonic.
			
			\begin{exem} The homomorphism $\xi=\xi_3(a)\circ\xi_2(b)\circ\xi_1(c):
				\left(\mathfrak{su}(2),\mathrm{diag}(\lambda_1,\mu_1,\nu_1)\right)\too\left(\mathfrak{su}(2),\mathrm{diag}(\lambda_2,\mu_2,\nu_2)\right)$ is biharmonic  not harmonic if
				\[ \mu_1=\nu_1,\mu_2=\nu_2\esp \cos(a)=\cos(b)=\left(\frac12  \right)^{\frac14}. \]
				
				\end{exem}

			\section{Harmonic and biharmonic homomorphisms of $\mathrm{sl}(2,\R)$}\label{section7}	
			Any homomorphism of $\mathrm{sl}(2,\R)$ is an automorphism and it is a product $\xi_3(a)\circ\xi_2(b)\circ\xi_1(c)$ where
			\[ \xi_1(a)=\left(\begin{array}{ccc}1&0&0\\0&\cosh(a) &\sinh(a)\\
			0&\sinh(a)&\cosh(a) \end{array}   \right),
			\xi_2(a)=\left(\begin{array}{ccc}\cosh(a)&0&\sinh(a)\\0&1 &0\\
			\sinh(a)&0&\cosh(a) \end{array}   \right),\xi_3(a)=\left(\begin{array}{ccc}\cos(a)&\sin(a)&0\\-\sin(a)& \cos(a)&0\\0&0&1 \end{array}   \right). \]
			If $\xi_i:(\mathfrak{su}(2),\prs_1)\too(\mathfrak{su}(2),\prs_2)$ with $\prs_j=\mathrm{Diag}(\la_j,\mu_j,\nu_j)$ then
			\[ \begin{cases}\di\tau(\xi_1(a))=-{\frac {\cosh \left( a \right) \sinh
					\left( a \right)  \left( \mu_{{2}}+\nu_{{2}} \right)  \left( \nu_{{1}
					}+\mu_{{1}} \right) }{\lambda_{{2}}\mu_{{1}}\nu_{{1}}}}X_1,
			\\
			\di\tau_2(\xi_1(a))=-2\,{\frac { \left( \mu_{{2}}+\nu_{{2}}
					\right) ^{2} \left(  \left( \cosh \left( a \right)  \right) ^{2}-1/2
					\right)  \left( \nu_{{1}}+\mu_{{1}} \right) ^{2}\cosh \left( a
					\right) \sinh \left( a \right) }{{\mu_{{1}}}^{2}{\nu_{{1}}}^{2}{
						\lambda_{{2}}}^{2}}}X_1
			\\
			\di\tau(\xi_2(a))={\frac {\cosh \left( a \right) \sinh
					\left( a \right)  \left( \lambda_{{2}}+\nu_{{2}} \right)  \left( \nu_
					{{1}}+\lambda_{{1}} \right) }{\mu_{{2}}\lambda_{{1}}\nu_{{1}}}}X_2,
			\\
			\di\tau_2(\xi_2(a))={\frac { \left( 2\, \left( \cosh \left( 
					a \right)  \right) ^{2}-1 \right) \cosh \left( a \right)  \left( 
					\lambda_{{2}}+\nu_{{2}} \right) ^{2} \left( \nu_{{1}}+\lambda_{{1}}
					\right) ^{2}\sinh \left( a \right) }{{\lambda_{{1}}}^{2}{\nu_{{1}}}^{
						2}{\mu_{{2}}}^{2}}}X_2,
			\\
			\di\tau(\xi_3(a))=-{\frac {\sin \left( a \right) \cos
					\left( a \right)  \left( \lambda_{{2}}-\mu_{{2}} \right)  \left( -\mu
					_{{1}}+\lambda_{{1}} \right) }{\nu_{{2}}\lambda_{{1}}\mu_{{1}}}}X_3,
			\\
			\di\tau_2(\xi_3(a))=-2\,{\frac {\sin \left( a \right) \cos
					\left( a \right)  \left( -\lambda_{{2}}+\mu_{{2}} \right) ^{2}
					\left( \mu_{{1}}-\lambda_{{1}} \right) ^{2} \left(  \left( \cos
					\left( a \right)  \right) ^{2}-1/2 \right) }{{\lambda_{{1}}}^{2}{\mu_
						{{1}}}^{2}{\nu_{{2}}}^{2}}}X_3.
			\\
			\end{cases} \]
			
		So we get:
		\begin{pr}\label{hsl2}\begin{enumerate}
				\item  $\xi_1(a)$ is biharmonic if and only if it is harmonic if only if $a=0$, i.e., $\xi_1=\mathrm{Id}$.
				\item  $\xi_2(a)$ is biharmonic if and only if it is harmonic if only if $a=0$, i.e., $\xi_2=\mathrm{Id}$.
				\item If $\la_2=\mu_2$ or $\la_1=\mu_1$ then $\xi_3(a)$ is harmonic.
				\item If $\la_2\not=\mu_2$ and $\la_1\not=\mu_1$ then $\xi_3(a)$ is harmonic if and only if $(\sin(2a)=0)$ and 
				$\xi_3(a)$ is biharmonic not harmonic if and only if $\cos(a)^2=\frac12$.			
			\end{enumerate}
			
		\end{pr}	
				
	\begin{theo}\label{thesl2} The automorphism
		\[ \xi=\xi_3(a)\circ\xi_2(b)\circ\xi_1(c):(\mathrm{sl}(2,\R),\mathrm{diag}(\la_1,\mu_1,\nu_1)\too (\mathrm{sl}(2,\R),\mathrm{diag}(\la_2,\mu_2,\nu_2), 0<\la_i\leq\mu_i,\nu_i>0 \]is harmonic if and only if $\xi_2(b)=\xi_1(c)=\mathrm{Id}_{\mathrm{sl}(2,\R)}$ and $\xi_3(a)$ is harmonic.
		
		\end{theo}	
		\begin{proof} We have
			\[\begin{cases}\di \tau(\xi)=\frac{(\mu_2+\nu_2)R}{\lambda_{{2}}\lambda_{{1}}\mu_{{1}}\nu_{{1}}}X_1+
			\frac{(\la_2+\nu_2)S}{\mu_{{2}}\lambda_{{1}}\mu_{{1}}\nu_{{1}}}X_2+\frac{(\la_2-\nu_2)Q}{\nu_{{2}}\lambda_{{1}}\mu_{{1}}\nu_{{1}}}X_3,\\
		\di	R=\cosh \left( b \right)  \left( \sinh \left( b \right) \lambda_{{1}}
			\sin \left( a \right)  \left( \mu_{{1}}+\nu_{{1}} \right)  \left( 
			\cosh \left( c \right)  \right) ^{2}-\sinh \left( c \right) \lambda_{{
					1}}\cos \left( a \right)  \left( \mu_{{1}}+\nu_{{1}} \right) \cosh
			\left( c \right) -\sinh \left( b \right) \nu_{{1}}\sin \left( a
			\right)  \left( \lambda_{{1}}-\mu_{{1}} \right)  \right),\\ 
		\di	S=\cosh \left( b \right)  \left( \sinh \left( b \right) \lambda_{{1}}
			\cos \left( a \right)  \left( \mu_{{1}}+\nu_{{1}} \right)  \left( 
			\cosh \left( c \right)  \right) ^{2}+\sinh \left( c \right) \lambda_{{
					1}}\sin \left( a \right)  \left( \mu_{{1}}+\nu_{{1}} \right) \cosh
			\left( c \right) -\sinh \left( b \right) \nu_{{1}}\cos \left( a
			\right)  \left( \lambda_{{1}}-\mu_{{1}} \right)  \right) ,\\
		\di	Q=-2\,\cosh \left( c \right) \sinh \left( b \right) \sinh \left( c
		\right) \lambda_{{1}} \left( \mu_{{1}}+\nu_{{1}} \right)  \left( \cos
		\left( a \right)  \right) ^{2}+\cosh \left( c
		\right) \sinh \left( b \right) \sinh \left( c \right) \lambda_{{1}}
		\left( \mu_{{1}}+\nu_{{1}} \right)\\
		\di+\sin \left( a \right)  \left( \lambda_
		{{1}} \left(  \left( \cosh \left( b \right)  \right) ^{2}-2 \right) 
		\left( \mu_{{1}}+\nu_{{1}} \right)  \left( \cosh \left( c \right) 
		\right) ^{2}-\nu_{{1}} \left( \lambda_{{1}}-\mu_{{1}} \right) 
		\left( \cosh \left( b \right)  \right) ^{2}+\lambda_{{1}} \left( \mu_
		{{1}}+\nu_{{1}} \right)  \right) \cos \left( a \right) .
			\end{cases} \]On the other hand, one can show easily
			\[ \begin{cases}
			\cos(a)R-\sin(a)S=-\cosh \left( b \right) \sinh \left( c \right) \cosh \left( c \right) 
			\lambda_{{1}} \left( \mu_{{1}}+\nu_{{1}} \right),\\ 
			\sin(a)R+\cos(a)S=\left( \lambda_{{1}} \left( \mu_{{1}}+\nu_{{1}} \right)  \left( \cosh
			\left( c \right)  \right) ^{2}+\nu_{{1}} \left( \mu_{{1}}-\lambda_{{1
				}} \right)  \right) \cosh \left( b \right) \sinh \left( b \right). 
					\end{cases} \]
			So $\xi$ is harmonic if and only if
			\[ \sinh(b)=\sinh(c)=Q=0  \]and we get the desired result.
			\end{proof}

\end{document}